\documentclass[11pt]{amsart} \textwidth=14.5cm \oddsidemargin=1cm
\usepackage{amsmath, amstext, amsbsy, amssymb, amscd, mathrsfs}
\usepackage{amsmath}
\usepackage{amsxtra}
\usepackage{amscd}
\usepackage{amsthm}
\usepackage{amsfonts}
\usepackage{amssymb}
\usepackage{eucal}
\usepackage{color}
\usepackage[all]{xy}
\usepackage{bbding}
\usepackage{pmgraph}
\usepackage{stmaryrd}
\usepackage{graphicx}
\usepackage{bm}
\usepackage[enableskew,vcentermath]{youngtab}

\newtheorem{theorem}{Theorem}[section]

\newtheorem{lemma}[theorem]{Lemma}
\newtheorem{proposition}[theorem]{Proposition}
\newtheorem{corollary}[theorem]{Corollary}
\newtheorem{example}[theorem]{Example}
\theoremstyle{definition}

\theoremstyle{remark}
\newtheorem{remark}[theorem]{Remark}

\definecolor{A}{rgb}{.75,1,.75}

\numberwithin{equation}{section}

\newcommand{\ds}{\displaystyle}

\newcommand{\C}{\mathbb C}
\newcommand{\Z}{\mathbb Z}

\newcommand{\K}{\mathbb K}
\newcommand{\Cl}{\mathcal C}
\newcommand{\HC}{\mathcal{H}_n^c}
\newcommand{\HCa}{{\mathcal{H}_{n,\bf A}^c}}
\newcommand{\HCK}{\mathcal{H}_{n,\K}^c}

\newcommand{\ga}{\gamma}
\newcommand{\Ga}{\Gamma}
\newcommand{\la}{\lambda}

\newcommand{\E}{\mathbb E}
\newcommand{\F}{\mathbb F}
\newcommand{\hf}{\frac12}
\newcommand{\mf}{\mathfrak}
\newcommand{\mc}{\mathcal}
\newcommand{\n}{\mathfrak{n}}
\newcommand{\h}{\mathfrak{h}}
\newcommand{\q}{\mathfrak{q}}
\newcommand{\g}{\mathfrak{g}}
\newcommand{\gl}{\mathfrak{gl}}
\newcommand{\ev}[1]{{#1}_{\bar{0}}}
\newcommand{\od}[1]{{#1}_{\bar{1}}}

\newcommand{\End}{{\rm End}}
\newcommand{\Hom}{{\rm Hom}}
\newcommand{\uni}{\underline{i}}
\newcommand{\sH}{\mathcal{H}_n^-}
\newcommand{\sHa}{\mathcal{H}_{n,\bf A}^-}
\newcommand{\sHK}{\mathcal{H}_{n,\K}^-}

\newcommand{\xn}{T_{w_{(n)}}}

\newcommand{\rr}{R}

{\vskip-\lastskip\medskip
  \noindent
  {\em #1.}\enspace
  }%
{\qed\par\medskip
  }

\begin{document}

\title
[Frobenius character formula and spin generic degrees]{Frobenius
character formula and spin generic degrees for Hecke-Clifford
algebra}
\author[Jinkui Wan and Weiqiang Wang]{Jinkui Wan and Weiqiang Wang}

\address{
Department of Mathematics,
Beijing Institute of Technology,
Beijing, 100081, P.R. China. }
\email{wjk302@gmail.com}

\address{Department of Mathematics, University of Virginia,
Charlottesville,VA 22904, USA.}
\email{ww9c@virginia.edu}

%\begin{abstract}
%The spin analogues of several classical concepts and results for
%Hecke algebras are established. A Frobenius type formula is obtained
%for irreducible characters of the Hecke-Clifford algebra $\HC$. The
%trace functions on $\HC$ are described and it leads to the notion of
%character table for $\HC$. The Hecke-Clifford algebra is endowed
%with a canonical symmetrizing trace form, with respect to which the
%spin generic degrees are formulated and shown to coincide with the
%spin fake degrees. We further characterize the trace functions and
%the symmetrizing trace form on the spin Hecke algebra which is
%Morita super-equivalent to $\HC$.
%\end{abstract}

\begin{abstract}
The spin analogues of several classical concepts and results for
Hecke algebras are established. A Frobenius type formula is obtained
for irreducible characters of the Hecke-Clifford algebra. A precise
characterization of the trace functions allows us to define the
character table for the algebra. The algebra is endowed with a
canonical symmetrizing trace form, with respect to which the spin
generic degrees are formulated and shown to coincide with the spin
fake degrees. We further provide a characterization of the trace
functions and the symmetrizing trace form on the spin Hecke algebra
which is Morita super-equivalent to the Hecke-Clifford algebra.
\end{abstract}

%\keywords{Hecke-Clifford algebras, Frobenius character formula}

\maketitle

\setcounter{tocdepth}{1}
 \tableofcontents

\section{Introduction}

\subsection{}

The Hecke algebras associated to symmetric groups or more general
finite Weyl groups are symmetric algebras with canonical
symmetrizing trace forms. The generic degrees defined in the
framework of generic Hecke algebras have played an important role in
finite groups of Lie type (first developed systematically by Lusztig
\cite{Lu}; also cf. Geck-Pfeiffer \cite{GP2}). On the other hand, a
Frobenius character formula for Hecke algebra associated to
symmetric groups has been established in \cite{Ram} (also see
King-Wybourne \cite{KW}), and the notion of character tables for
Hecke algebras has been formulated by Geck and Pfeiffer \cite{GP1}.

The Hecke-Clifford algebra $\HC$, which is a deformation of the
algebra $\mf H^c_n =\Cl_n \rtimes S_n$, admits a natural
superalgebra structure, first appeared in Olshanski \cite{Ol} who
formulated a super queer version of the Schur-Jimbo duality. Its
representation theory was subsequently developed by Jones-Nazarov
\cite{JN} for a generic quantum parameter $v$, and it is indeed
closely related to the spin representations of the symmetric group
developed by Schur \cite{Sch}. In particular the irreducible
characters $\zeta^\la$ of $\HC$ (always understood in the
$\Z_2$-graded sense in this paper) are parametrized by the strict
partitions of $n$. This algebra $\HC$ is also known to be Morita
super-equivalent to a spin Hecke algebra $\sH$ introduced by the
second author \cite{W}, which is a deformation of Schur's spin
symmetric group algebra.

\subsection{}

Here is a quick summary of the main results of this paper. We first
establish a Frobenius type formula for irreducible characters of the
Hecke-Clifford algebra. We endow the Hecke-Clifford algebra with a
canonical symmetrizing trace form $\gimel$, and then find an
explicit shifted hook formula for the spin generic degrees for the
Hecke-Clifford algebra by a novel and simple application of our
Frobenius type character formula. The spin Hecke algebra is also
shown to carry a natural symmetrizing trace form, which is
compatible with $\gimel$ for Hecke-Clifford algebra via the Morita
super-equivalence.

\subsection{}

Let us describe in some detail. Our approach to obtaining a
Frobenius type character formula for Hecke-Clifford algebra
(Theorem~\ref{thm:Frobenius}) takes advantage of the
Sergeev-Olshanski duality, and it is inspired by Ram's approach who
obtained a Frobenius character formula for the type $A$ Hecke
algebra via the Schur-Jimbo duality. The symmetric functions arising
in our Frobenius type formula are certain spin Hall-Littlewood
functions introduced  by the authors in \cite{WW2}, which are
one-parameter deformation of Schur $Q$-functions. This should be
compared to the appearance of Hall-Littlewood functions in
\cite{Ram}.

We show that every trace function on the Hecke-Clifford algebra over
the ring ${\bf A}=\Z[\hf][v,v^{-1}]$ is completely determined by its
values on the standard elements parametrized by the odd partitions
of $n$ (see Theorem~\ref{thm:spacetrace} and
Corollary~\ref{cor:spacetrace}). This leads to well-defined notions
of class polynomials and character table for Hecke-Clifford algebra,
similar to those for Hecke algebras introduced by Geck-Pfeiffer
\cite{GP1}.

The Hecke-Clifford algebra is a symmetric superalgebra endowed with
a canonical symmetrizing trace form $\gimel$ (the choice of $\gimel$
is not obvious as it does not restrict to the well-known
symmetrizing trace form on its type $A$ Hecke subalgebra), and this
allows us to formulate the {\em spin generic degrees} $D^\la$ for
the irreducible characters $\zeta^\la$ of $\HC$. Recall the authors
\cite{WW1} formulated earlier a spin coinvariant algebra for the
algebra $\mf H^c_n$, and found a closed formula for the so-called
{\em spin fake degrees} (this terminology appeared later in
\cite{WW3}).
%The spin fake degrees for an arbitrary Weyl group can also be
%naturally formulated, and their computations in all type is the
%subject of a forthcoming Virginia dissertation of Constance Baltera.
For a symmetric algebra $\mc H$ with a symmetrizing form,
there exist elements called {\em Schur elements} (cf. \cite[Theorem 7.2.1]{GP2})
for irreducible characters, which can be used to determine when $\mc H$ is semsimple.
These elements are closely related to the generic degrees
in the case of usual Hecke algebras (cf. \cite[Section 8.1.8]{GP2}).
The Schur elements for Hecke-Clifford algebra are computed
explicitly (Theorem~\ref{th:spinSchur}), and they do not lie in $\bf
A$ in general. The spin generic degrees for Hecke-Clifford algebra
are shown to be polynomials in the quantum parameter $v$ and they
match perfectly with the spin fake degrees
(Theorem~\ref{th:genDeg}). This phenomenon is strikingly parallel to
the classical result due to Steinberg \cite{S} that the generic
degrees for the type $A$ Hecke algebras coincide with the fake
degrees for symmetric groups (also cf. \cite{Lu, GP2}).

%Though various connections between characters and generic degrees of
%Hecke algebras have been explored in literature, our approach of
%deriving the closed formula for $D^\la$ directly from the Frobenius
%character formula is quite elegant and seems to be new even in the
%usual Hecke algebra setting. We hope to apply the same strategy
%elsewhere to revisit the generic degrees (or more general notion of
%weights) for Hecke algebras.

We also succeed (see Theorem~\ref{thm:spintrace}) in describing the
space of trace functions of the spin Hecke algebra $\sH$ introduced
in \cite{W}. The canonical trace form $\gimel^-$ on $\sH$
corresponding to the form $\gimel$ on $\HC$ under the Morita
super-equivalence is characterized in a simple way
(Theorem~\ref{thm:sHtrace0}), in spite of the fact that the braid
relation is deformed for $\sH$ and the standard elements depend on
the choices of reduced expressions of a given element.

\subsection{}

Here is the layout of the paper. In Section~\ref{sec:duality}, we
review the Sergeev duality and Olshanski duality, and set up various
notations needed in the remainder of the paper. In
Section~\ref{sec:char},  using the Olshanski duality we compute a
Frobenius type character formula for $\HC$, whose specialization at
$v=1$ is equivalent to the classical character formula of Schur for
spin symmetric groups. In Section~\ref{sec:trace}, by a sequence of
reductions we show that the trace functions are determined
by their values on standard elements parametrized by odd partitions of
$n$. In Section~\ref{sec:degree}, the symmetrizing trace form
$\gimel$ on Hecke-Clifford algebra is introduced, and the spin
generic degrees are computed using the Frobenius character formula
given in Section~\ref{sec:char}. Finally in Section~\ref{sec:spinH},
we describe the counterparts of Section~\ref{sec:trace} and part of
Section~\ref{sec:degree} for spin Hecke algebras.

{\bf Acknowledgements.}  The first author was partially supported by
NSFC-11101031, and she thanks Shun-Jen Cheng at Academia Sinica for support
and providing an excellent atmosphere in the summer of 2011, where
part of this paper was written. The second author was partially
supported by NSF DMS-1101268.

\section{The Sergeev-Olshanski duality}
\label{sec:duality}

In this preliminary section, we shall introduce the Hecke-Clifford
algebra, review the Sergeev-Olshanski duality, and set up notations
to be used in later sections.

\subsection{Basics on superalgebras}\label{sec:super}

Let $\F$ be a field, which is always assumed to be of characteristic
not equal to $2$ in this paper. By a vector superspace over $\F$ we
mean a $\Z_2$-graded space $V =\ev V \oplus \od V$. If $\dim
\ev{V}=r$ and $\dim \od{V}=m$, we write $\underline{\dim}V=r|m$.
Given a homogeneous element $0\neq v\in V$, we denote its degree by
$|v|\in\Z_2$. An associative $\F$-superalgebra $\mc A =\ev{\mc{A}}
\oplus \od{\mc{A}}$ satisfies $\mc{A}_i \cdot \mc{A}_j \subseteq
\mc{A}_{i+j}$ for $i,j \in\Z_2$. By an ideal $I$ and respectively a
module $M$ of a superalgebra $A$, we  always mean that $I$ and $M$
are $\Z_2$-graded, i.e., $I=(I\cap \ev{\mc{A}})\oplus (I\cap
\od{\mc{A}})$, $M=\ev M \oplus \od M$ such that $A_i M_j \subseteq
M_{i+j}$ for $i,j \in \Z_2$.  The superalgebra $\mc A$ is called
simple if it has no non-trivial ideals.

Let $V$ be an $\F$-superspace with $\underline{\dim}V=r|m$, then
$$
M(V):=\End_{\F}(V)
$$
is a simple superalgebra.
%, which is isomorphic to the algebra
%$M(r|m)$ consisting of $(r+m)\times (r+m)$ matrices. Moreover $M(V)$
%is a simple superalgebra.
Assume now in addition $r=m$ and an odd automorphism $J$ of $V$ of
order $2$ is given. The subalgebra of $\End_{\F} (V)$,
$$
Q(V) =\{ x \in \End_{\F} (V) \mid x \text{ and } J\text{
super-commute}\},
$$
is also a simple superalgebra.
Observe that the resulting superalgebras $Q(V)$ are isomorphic to each other
for different automorphisms $J$.
%Later on in the paper, our choice of $J$ may vary in different situations for convenience.
%$J$ is the linear transformation in the block matrix form
%\begin{equation*}
%\begin{pmatrix}
%0&I_m \\
%I_m &0\\
%\end{pmatrix},
%\end{equation*}
%we write $Q(V)$ as $Q(m)$, which consists of $2m\times 2m$ matrices
%of the form:
%\begin{equation*}
%\begin{pmatrix}
%a&b\\
%-b&a\\
%\end{pmatrix},
%\end{equation*}
%where $a$ and $b$ are arbitrary $m\times m$ matrices. Observe that
%we have a superalgebra isomorphism $Q(V) \cong Q(m)$ if
%$\underline{\dim} V =m|m$.

An irreducible module $V$ over an $\F$-superalgebra $\mc A$ is
called {\em split irreducible} if $\E\otimes_{\F}V$ is irreducible
over $\E\otimes_{\F}\mc A$ for any field extension $\E\supseteq \F$.
A split irreducible $\mc A$-module $V$ is of type $\texttt M$ if
$\End_{\F} (V)$ is one-dimensional and of type $\texttt Q$ if
$\End_{\F} (V)$ is two-dimensional. A superalgebra $\mc A$ is {\em
split semisimple} if $\mc A$ is a direct sum of simple algebras of
the form $M(V)$ and $Q(V)$ for various $V$.

Recall that given two superalgebras $\mathcal{A}$ and $\mathcal{B}$,
the tensor product  $\mathcal{A}\otimes\mathcal{B}$ is naturally a
superalgebra, with multiplication defined by
\begin{equation}  \label{eq:superotimes}
(a\otimes b)(a'\otimes b')=(-1)^{|b| \cdot |a'|}(aa')\otimes (bb')
\qquad (a,a'\in\mathcal{A}, b,b'\in\mathcal{B}).
\end{equation}
The following lemma can be found in \cite{Jo}
%or \cite[Lemma~ 12.2.13]{Kle}
(where $\F$ is assumed to be an algebraically closed field).

\begin{lemma}\label{tensorsmod}
Let $V$ be a split irreducible $\mathcal{A}$-module and $W$ be a
split irreducible $\mathcal{B}$-module.
\begin{enumerate}
\item If both $V$ and $W$ are of type $\texttt{M}$, then
$V\otimes W$ is a split irreducible
$\mathcal{A}\otimes\mathcal{B}$-module of type $\texttt{M}$.

\item If one of $V$ or $W$ is of type $\texttt{M}$ and the other
is of type $\texttt{Q}$, then $V\otimes W$ is a split irreducible
$\mathcal{A}\otimes\mathcal{B}$-module of type $\texttt{Q}$.

\item If both $V$ and $W$ are of type $\texttt{Q}$, then
$V\otimes W$ is a sum of two isomorphic copies of a split
irreducible module of type $\texttt M$, which will be denoted by
$2^{-1} V\otimes W$.
\end{enumerate}
Moreover, all split irreducible
$\mathcal{A}\otimes\mathcal{B}$-modules arise as components of
$V\otimes W$ for some choice of irreducibles $V,W$.
\end{lemma}

\subsection{The Sergeev duality}

In this subsection, we shall take $\F=\C$. Denote by $\Cl_n$ the
Clifford superalgebra generated by odd elements $c_1,\ldots,c_n$
subject to the relations
\begin{equation}  \label{eq:Cl}
c_i^2=1,c_ic_j=-c_jc_i, \quad 1\leq i\neq j\leq n.
\end{equation}
Denote by $\mf{H}^c_n=\mathcal{C}_n\rtimes S_n$ the superalgebra
generated by the even elements $s_1,\ldots,s_{n-1}$ and the odd
elements $c_1,\ldots,c_n$  subject to \eqref{eq:Cl}, the standard
Coxeter relation among $s_i=(i,i+1)$ for the symmetric group $S_n$,
and the additional relations:
\begin{align*}
%s_i^2&=1, s_is_j=s_js_i,\quad 1\leq i,j\leq n-1, |i-j|>1,\\
%s_is_{i+1}s_i&=s_{i+1}s_is_{i+1},\quad 1\leq i\leq n-2,\\
s_ic_i&=c_{i+1}s_i, s_ic_j=c_js_i,\quad 1\leq i,j\leq n-1, j\neq i,i+1.
\end{align*}

Denote by $\mc{P}_n$  the set of all partitions of $n$ and by
$\mc{CP}_n$ the set of compositions of $n$. Let $\mc{SP}_n$
(respectively, $\mc{OP}_n$) denote the set of strict (respectively,
odd) partitions of $n$. For $\la\in\mc{P}_n$, denote by $\ell(\la)$
the length of $\la$ and let
\begin{align*}
\delta(\la)=\left\{
\begin{array}{ll}
0,&\text{ if }\ell(\la)\text{ is even},\\
1,&\text{ if } \ell(\la)\text{ is odd}.
\end{array}
\right.
\end{align*}
Denote by
$Q_\la$ the Schur $Q$-functions associated to a strict partition
$\la$ (cf. \cite{Mac, WW3}).
It is known \cite{Jo, Se} that there exists a characteristic map
(cf. \cite[(3.12)]{WW3})
relating the representation theory of the algebra
$\mf{H}^c_n$ to the theory of symmetric functions,
which can be viewed as a analog of the  Frobenius characteristic map
in the representation theory of symmetric groups.
More precisely, for each strict partition $\la$ of
$n$, there exists an irreducible $\mf H^c_n$-module $U^{\la}_1$
which corresponds to the Schur $Q$-function $Q_\la$ (up to some
$2$-power) under the characteristic map, and $\{U^{\la}_1 \mid
\la\in\mc{SP}_n\}$ forms a complete set of non-isomorphic
irreducible $\mf H^c_n$-modules. Furthermore, $U^{\la}_1$ is of type
$\texttt M$ if $\delta(\la)=0$ and  is of type $\texttt Q$ if
$\delta(\la)=1$. Denote by $\zeta_1^{\la}$ the character of
$U^{\la}_1$ for $\la\in\mc{SP}_n$.

The queer Lie superalgebra, denoted by $\q(m)$, is the Lie
superalgebra associated to the associative superalgebra $Q(m)$ with
respect to the super-bracket. For convenience, in the case when
$\F=\C$, we shall take the odd involution
\begin{equation}\label{eq:matrixP}
P=\sqrt{-1}\begin{pmatrix}
0&I_m \\
-I_m &0\\
\end{pmatrix},
\end{equation}
then $Q(m)$ and hence $\q(m)$ will consist of
$2m\times 2m$
matrices of the form:
\begin{equation}\label{eqn:qmatrix}
\begin{pmatrix}
a&b\\
b&a\\
\end{pmatrix},
\end{equation}
where $a$ and $b$ are arbitrary $m\times m$ matrices over $\C$,
and the rows and columns of \eqref{eqn:qmatrix} are labeled by the
set
$$
I(m|m) :=\{-1, \ldots, -m,1,\ldots, m\}.
$$
Let $\g=\q(m)$. The even (respectively, odd) part $\ev\g$
(respectively, $\od\g$) consists of those matrices of the form
\eqref{eqn:qmatrix} with $b=0$ (respectively, $a=0$). Denote by
$E_{ij}$ for $i,j \in I(m|m)$ the standard elementary matrix with
the $(i,j)$th entry being $1$ and zero elsewhere. Fix the triangular
decomposition $ \g=\n^-\oplus\h\oplus\n^+, $ where $\h$
(respectively, $\n^+$, $\n^-$) is the subalgebra of $\g$ which
consists of matrices of the form ~(\ref{eqn:qmatrix})~ with $a, b$
being arbitrary diagonal (respectively, upper triangular, lower
triangular) matrices. Let $\mf{b}=\h\oplus\n^+$. We denote the
standard basis for $\ev\h$ by
$$
H_i =E_{-i,-i} +E_{ii}, \qquad 1\le i \le m.
$$
%Let $\{\epsilon_i \mid i=1, \ldots, m\}$ be the basis dual to the
%standard basis $\{H_i \mid i=1, \ldots, m\}$ for the even subalgeba
%$\ev \h$ of $\h$, where $\ev \h$ is identified with the standard
%Cartan subalgebra of $\mf{gl}(m)$ via the natural isomorphism
%$\ev{\q(m)} \cong \mf{gl}(m)$. With respect to $\ev\h$ we have the
%root space decomposition
%$\g=\h\oplus\oplus_{\alpha\in\Gamma}\g_\alpha$ with roots
%$\Gamma=\{\epsilon_i-\epsilon_j|1\le i\not=j\le m\}.$ The positive
%system corresponding to the Borel subalgebra $\mf b$ is
%$\Gamma^+=\{\epsilon_i-\epsilon_j|1\le i<j\le m\}.$

Every finite-dimensional $\g$-module is isomorphic to a highest
weight module $V_1(\la)$ generated by a vector $v_\la$ satisfying
$\n^+.v_\la=0$ and $hv_\la=\la(h)v_\la$ for $h \in\ev\h$, for some
$\la \in \ev\h^*$. We have a weight space decomposition $V_1(\la)
=\oplus_{\mu} V_1(\la)_\mu$, where a weight $\mu$ can be identified
with an $m$-tuple $(\mu_1, \ldots, \mu_m)$. Let $x_1,\ldots, x_m$ be
$m$ independent variables. A character of a $\q(m)$-module with
weight space decomposition $M=\oplus M_{\mu}$ is defined to be
$$
{\rm ch} M=\sum_{\mu=(\mu_1,\ldots,\mu_m)}\dim M_{\mu}
x_1^{\mu_1}\cdots x_m^{\mu_m}.
$$

We have a representation $(\omega_n, (\C^{m|m})^{\otimes
n})$ of $\gl(m|m)$, hence of its subalgebra $\q(m)$, and we also have
a representation $(\psi_n, (\C^{m|m})^{\otimes n})$ of the algebra $\mf H^c_n$
 defined by
\begin{align*}
\psi_n (s_i)&.(v_1\otimes  \ldots\otimes v_i\otimes v_{i+1}\otimes\ldots \otimes v_n) =
(-1)^{|v_i|\cdot|v_{i+1}|}v_1\otimes
 \ldots\otimes v_{i+1}\otimes v_i\otimes\ldots \otimes v_n,\\
 \psi_n (c_i)&.(v_1\otimes  \ldots \otimes v_n) =
 (-1)^{(|v_1| +\ldots +|v_{i-1}|)} v_1\otimes
   \ldots \otimes v_{i-1} \otimes P v_i \otimes\ldots \otimes v_n,
\end{align*}
where $v_i, v_{i+1} \in \C^{m|m}$ are $\Z_2$-homogeneous. We recall a classical result of
Sergeev.

\begin{proposition}   \cite[Theorem 3]{Se}
 \label{prop:Sergeev}
The algebras $\omega_n(U(\q(m)))$ and $\psi_n(\mf H^c_n)$ form mutual
centralizers in $\End_\C ((\C^{m|m})^{\otimes n})$. As an $U(\q(m))
\otimes \mf{H}^c_n$-module, we have
\begin{align*}
V^{\otimes n}\cong\bigoplus _{\la \in\mc{SP}_n, \ell(\la)\leq
m}2^{-\delta(\la)}V_1(\la)\otimes U^{\la}_1.
\end{align*}
Moreover, the character of $V_1(\la)$  is given by
\begin{equation}\label{eqn:queer.ch}
{\rm ch}V_1(\la)=2^{-\frac{\ell(\la)-\delta(\la)}{2}}Q_{\la}(x_1, \ldots, x_m).
\end{equation}
\end{proposition}

%Let $e_{-1},\ldots, e_{-m}, e_1,\ldots, e_m$ be the natural basis of $\C^{m|m}$.
%Then the tensor space $V^{\otimes n}$ has a basis given by
%$\{e_{i_1}\otimes\cdots\otimes e_{i_n}|i_1,\ldots, i_n\in I(m|m)\}$.
%For a composition $\ga=(\ga_1,\ga_2,\ldots, \ga_m)\in\mc{CP}_n$ ,
%we define a projection operator %$E_{\ga}$ on $V^{\otimes n}$ by
%\begin{align*}
%E_{\ga}(e_{i_1}\otimes\cdots\otimes e_{i_n})= \left \{
 %\begin{array}{ll}
%e_{i_1}\otimes\cdots\otimes e_{i_n} ,
% & \text{ if } \sharp\{i_k||~i_k|=j, 1\leq k\leq n\}=\ga_j, 1\leq j\leq m, \\
 %0
 %, & \text{ otherwise}.
 %\end{array}
 %\right.
%end{align*}
Associated to the formal variables $x_1,\ldots, x_m$, let $D$ be the operator defined by
\begin{equation}  \label{eq:D}
D =x_1^{H_1} \cdots x_m^{H_m}.
%D=\sum_{\ga=(\ga_1,\ldots,\ga_m)\in\mc{CP}_n}x_1^{\ga_1}\ldots x_m^{\ga_m}E_{\ga}.
\end{equation}
The operator $D$ commutes with the action of $\mf H^c_n$ on
$(\C^{m|m})^{\otimes n}$, since the action of $\mf H^c_n$ preserves
the weight space decomposition. Proposition~\ref{prop:Sergeev} has
the following corollary.

\begin{corollary}\label{cor:Frobenius}
For $h\in \mf H^c_n$, the trace of the linear operator $Dh$ on
$(\C^{m|m})^{\otimes n}$ is
$$
\text{tr}(Dh)=\sum_{\la \in\mc{SP}_n, \ell(\la)\leq
m}2^{-\frac{\ell(\la)+\delta(\la)}{2}}Q_{\la}(x_1,\ldots,
x_m)\zeta^{\la}_1(h).
$$
\end{corollary}

\subsection{The Hecke-Clifford algebra $\HC$.}

Let $v$ be an indeterminate. The Hecke-Clifford algebra $\HC$ is the
associative superalgebra over the field $\C(v^{\frac 12})$ with the
even generators $T_1,\ldots,T_{n-1}$ and the odd generators
$c_1,\ldots,c_n$ subject to the following relations:
\begin{align}
(T_i-v)(T_i+1)=0, & \quad 1\leq i\leq n-1,\notag\\
T_iT_j=T_jT_i, &  \quad 1\leq i,j\leq n-1, |i-j|>1,\notag\\
T_iT_{i+1}T_i =T_{i+1}T_iT_{i+1}, & \quad 1\leq i\leq n-2,\notag\\
 c_i^2=1,c_ic_j=-c_jc_i, & \quad 1\leq i\neq j\leq n,\notag\\
T_ic_j =c_jT_i, & \quad j\neq i,i+1, 1\leq i\leq n-1, 1\leq j\leq n ,\notag\\
T_ic_i=c_{i+1}T_i, & \quad
%T_ic_{i+1}=c_iT_i-(v-1)(c_i-c_{i+1}), \quad
  1\leq i\leq n-1. \notag
\end{align}

Note that  the specialization of $\HC$ at $v=1$ recovers $\mf
H^c_n$. Denote
$$
[n] :=\{1, \ldots, n\}.
$$
For an (ordered) subset $I =\{i_1, i_2, \ldots, i_k\} \subseteq
[n]$,  we denote $C_I =c_{i_1}c_{i_2} \ldots c_{i_k}$. By
convention, $C_\emptyset =1$. Then $\{C_I|~I\subseteq[n]\}$ is a
basis for the Clifford algebra $\Cl_n$. According to
\cite[Proposition 2.1]{JN}, the set $\{T_{\sigma}C_I \mid \sigma\in
S_n, I\subseteq [n]\}$ forms a linear basis of the algebra $\HC$.

\begin{remark}\label{rem:twoHC}
Our definition of the Hecke-Clifford algebra $\HC$ is slightly
different from the algebra introduced in \cite{Ol}, where the
quantum parameter $q$ is used. The even generators $t_1,\ldots,
t_{n-1}$ in \cite{Ol}  are related to $T_1,\ldots, T_{n-1}$ via
$v=q^2$ and $t_i=v^{-\frac 12}T_i$, for $1\leq i\leq n-1$.
\end{remark}

Write $ [k]_{v}=\ds\frac{v^k-v^{-k}}{v-v^{-1}} $ for $k\in\Z_+$. Let
$$
\K :=\C(v^{\frac 12}) \big(\sqrt{[1]_v},\ldots,\sqrt{[n]_v} \big)
$$
be the field extension of $\C(v^{\frac 12})$, and denote
$\HCK:=\K\otimes_{\C(v^{\frac 12})}\HC$.

\begin{proposition}\cite[Corollary 6.8]{JN}
  \label{prop:JN}
The $\K$-superalgebra $\HCK$ is split semisimple. For each
$\la\in\mc{SP}_n$, there exists an irreducible representation
$U^{\la}$ of $\HCK$ with character $\zeta^{\la}$ such that
$\{U^{\la} \mid \la\in\mc{SP}_n\}$ forms a complete set of
nonisomorphic irreducible $\HCK$-modules. Moreover, the
specialization at $v=1$ of $\zeta^{\la}$  gives $\zeta^{\la}_1$ for
$\la\in\mc{SP}_n$.
\end{proposition}

\begin{remark}
The precise form of the field $\K$ is not relevant in this paper and
we could simply take $\K$ to be the algebraic closure of
$\C(v^{\frac 12})$ as well. The construction of the irreducible
$\HCK$-modules $U^{\la}$ in \cite{JN} can be streamlined by a
straightforward $v$-analogue of the semiformal form construction of
the irreducible $\mf H^c_n$-modules given in \cite{HKS} and independently in
\cite{Wan} (cf. \cite[Section~7]{WW3}). One advantage of the
latter approach is that it works equally well for the affine
Hecke-Clifford algebra.
\end{remark}

\subsection{The action of $\HC$ on $V^{\otimes n}$}

Let $V=\K^{m|m}$ be the vector $\K$-superspace with
$\underline{\dim}_\K V=m|m$. The standard basis $e_{-m},\ldots,
e_{-1}, e_1,\ldots, e_m$ of $\C^{m|m}$ can be naturally regarded as
a $\K$-basis of $V$ and  the operator $D$ is defined on $V^{\otimes
n}$. Meanwhile $\{E_{ij} \mid i,j\in I(m|m)\}$ can be regarded as
the standard basis of $\End_\K(V)$ with respect to the basis
$\{e_{-m},\ldots, e_{-1},e_1,\ldots, e_m\}$ of $V$, i.e.,
$E_{ij}(e_k)=\delta_{jk}e_i$ for $k\in I(m|m)$. Following \cite{Ol},
we set
\begin{align}
\Theta&= \sqrt{-1} \sum_{1\leq a\leq m}(E_{-a,a}-E_{a,-a}),\notag\\
Q&=  \sum_{i,j\in I(m|m)}\text{sgn}(j)E_{ij}\otimes E_{ji},\notag\\
S&= v^{\hf}\sum_{i\leq j\in I(m|m)}S_{ij}\otimes E_{ij}\in\End_\K
(V^{\otimes 2}),
 \label{eq:operatorS}
\end{align}
where $S_{ij}$ are defined as follows:
\begin{align}
S_{aa}&= 1+(v^{\frac12}-1)(E_{aa}+E_{-a,-a}),\quad 1\leq a\leq m,\label{eq:S1}\\
S_{-a,-a}&=1+(v^{-\frac{1}{2}}-1)(E_{aa}+E_{-a,-a}),\quad1\leq a\leq m, \label{eq:S2}\\
S_{ab}&=(v^\hf -v^{-\hf})(E_{ba}+E_{-b,-a}),\quad 1\leq a<b\leq m, \label{eq:S3}\\
S_{-b,-a}&=-(v^\hf -v^{-\hf})(E_{ab}+E_{-a,-b}),\quad 1\leq a<b\leq m, \label{eq:S4}\\
S_{-b,a}&=-(v^\hf -v^{-\hf}) (E_{-a,b}+E_{a,-b}),\quad 1\leq a, b\leq m. \label{eq:S5}
\end{align}
We remark that our definition of the operator $S$ is  a  $v^{\frac
12}$ multiple of the original operator defined in \cite{Ol} (see
Remark~\ref{rem:twoHC}, \eqref{eq:OmeS} and
Proposition~\ref{prop:Ol} below).

To endomorphisms $A\in\End_\K(V)$ and
$C=\sum_{\alpha}A_{\alpha}\otimes B_{\alpha}\in\End_\K(V^{\otimes
2})$, we associate the following elements in $\End_\K(V^{\otimes
n})$:
\begin{align*}
A^{k}&={\rm I}^{\otimes k-1}\otimes A\otimes {\rm I}^{n-k},\\
C^{j,k}&=\sum_{\alpha}A_{\alpha}^{j}B_{\alpha}^k, \qquad 1\leq j\neq k\leq n.
\end{align*}

Olshanski \cite{Ol} introduced the quantum deformation $U_v(\q(m))$
of the universal enveloping algebra of $\q(m)$, which is a
$\K$-algebra with generators $L_{ij}$ for $i, j\in I(m|m)$ with
$i\le j$ subject to certain explicit relations (which we do not need
here).
%
%The quantum superalgebra $U_v(\q(m))$ is defined to be the associative algebra over $\C(v^{\frac %12})$
% generated by $L_{ij}$ with $  i\leq j\in I(m|m)$ subject to the relations:
% \begin{align*}
% L_{ii}L_{-i,-i}&=L_{-i,-i}L_{ii}=1,\\
%(-1)^{p(i,j)p(k,\ell)}&v^{\frac{\varphi(j,\ell)}{2}}L_{ij}L_{k\ell}
%+\{k\leq j<\ell\}\theta(i,j,k)(v^{\frac12}-v^{-%\frac12})L_{i\ell}L_{kj}\\
%& +\{i\leq-\ell<j\leq -k)\theta(-i,-j,k\}(v^{\frac12}-v^{-\frac12})L_{i,-\ell}L_{k,-j}\\
%=&v^{\frac{\varphi(i,k)}{2}}L_{k\ell}L_{ij}+\{k<i\leq\ell\}
%\theta(i,j,k)(v^{\frac12}-v^{-\frac12})L_{i\ell}%L_{kj}\\
%&+\{-\ell\leq i<-k\leq j\}\theta(-i,-j,k)(v^{\frac12}-v^{-\frac12})L_{-i,\ell}L_{-k,j},
% \end{align*}
% where $\varphi(i,j)=\delta_{|i|,|j|}{\rm sgn}(j),
%\theta(i,j,k)={\rm sgn}({\rm sgn}(i)+{\rm sgn}(j)+{\rm sgn}%(k))$,
 %$p(i,j)=0$ if $ij>0$ and $p(i,j)=1$ if $ij<0$ for any $i,j,k\in I(m|m)$,
 %and the simbol $\{\cdots\}$ (the %dots stand for some inequalities),
%  is equal to 1 if all these inequalities are fulfilled and $0$ otherwise.
%
Define $\Omega_n:U_v(\q(m))\rightarrow \End_\K (V^{\otimes n})$ by
letting
\begin{equation}  \label{eq:OmeS}
  \Omega_n(L_{ij})= S_{ij}, \qquad \text{ for } i\leq j\in I(m|m).
\end{equation}
Then it is known \cite{Ol} that $\Omega_n$ is an algebra
homomorphism and hence defines a representation $(\Omega_n,
V^{\otimes n})$ of $U_v(\q(m))$, which is a deformation of the
representation $(\omega_n, (\C^{m|m})^{\otimes n})$.

\begin{proposition}\cite[Theorems 5.2, 5.3]{Ol}
    \label{prop:Ol}
Let $\bar{S}=QS\in\End_\K (V^{\otimes 2})$. Then there exists a
representation $(\Psi_n, V^{\otimes n})$ of the Hecke-Clifford
algebra $\HCK$ defined by
$$
\Psi_n(T_j)=\bar{S}^{j,j+1},\quad \Psi_n(c_k)=\Theta^{k},
$$
where $1\leq j\leq n-1$ and $1\leq k\leq n$.
The algebras $\Omega_n(U_v(\q(m)))$ and $\Psi_n(\HCK)$ form mutual
centralizers in $\End_\K (V^{\otimes n})$. Moreover,  as a
$U_v(\q(m))\otimes\HCK$-module, we have a multiplicity-free
decomposition
$$
V^{\otimes n}\cong \bigoplus_{\la\in\mc{SP}_n,\ell(\la)\leq
m}2^{-\delta(\la)}V(\la)\otimes U^\la,
$$
where $V(\la)$'s are pairwise non-isomorphic irreducible
$U_v(\q(m))$-modules.
 \end{proposition}

The operator $D$ in \eqref{eq:D} commutes with the action of $\HCK$
on $V^{\otimes n}$, since the action of $\HCK$ preserves the weight
space decomposition. We have the following generalization of
Corollary~\ref{cor:Frobenius}.

\begin{proposition}\label{prop:generaltrace}
For  $h\in\HCK$, we have
$$
{\rm tr}(Dh)=\sum_{\la\in\mc{SP}_n, \ell(\la)\leq
m}2^{-\frac{\ell(\la)+\delta(\la)}{2}}Q_{\la}(x_1,\ldots,
x_m)\zeta^{\la}(h).
$$
\end{proposition}

\begin{proof}
Fix $\mu=(\mu_1,\ldots,\mu_m)\in\mc{CP}_n$. Then the weight subspace
$(V^{\otimes n})_{\mu}$ of $V^{\otimes n}$ has a basis given by
$e_{i_1}\otimes\cdots\otimes e_{i_n}$ with $i_1,\ldots,
i_n\in I(m|m)$ and $\sharp\{j||i_j|=k\}=\mu_k$ for $1\leq k\leq m$. The
operator $D$ acts on $V^{\otimes n}_{\mu}$ as $(x_1^{\mu_1}\cdots
x_m^{\mu_m})\cdot{\rm I}$. On the other hand, $(V^{\otimes
n})_{\mu}$ is stable under the action of $\HCK$ and hence it can be
decomposed as the direct sum of the irreducible module $U^{\la}$.
This implies that the trace of $Dh$ on $(V^{\otimes n})_{\mu}$ can
be written as
$$
{\rm tr} Dh|_{(V^{\otimes n})_{\mu}}=x_1^{\mu_1}\cdots
x_m^{\mu_m}\sum_{\la\in\mc{SP}_n}f_{\la\mu}\zeta^{\la}(h)
$$
for some $f_{\la\mu}\in\Z_+$.
This holds for all $\mu=(\mu_1,\ldots,\mu_m)\in\mc{CP}_n$ and hence
$$
{\rm tr}(Dh)=\sum_{\la\in\mc{SP}_n}f_\la(x_1,\ldots, x_m) \zeta^{\la}(h),
$$
where $f_{\la}(x_1,\ldots,
x_m)=\sum_{\mu=(\mu_1,\ldots,\mu_m)\in\mc{CP}_n}f_{\la\mu}x^{\mu_1}_1\ldots
x^{\mu_m}_m$. Specializing $v=1$ and using
Proposition~\ref{prop:JN}, we obtain that, for $h\in\mf H^c_n$,
$$
{\rm tr}(Dh)=\sum_{\la\in\mc{SP}_n}f_\la(x_1,\ldots, x_m) \zeta^{\la}_1(h).
$$
Comparing with Corollary~\ref{cor:Frobenius} and noting the linear
independence of irreducible characters $\zeta_1^{\la}$ for $\la \in
\mc{SP} _n$, one can deduce that
 $
 f_{\la}(x_1,\ldots, x_m)=Q_{\la}(x_1,\ldots,x_m)
 $
if $\ell(\la)\leq m$ and $f_{\la}(x_1,\ldots, x_m)=0$ otherwise.
\end{proof}

\section{A Frobenius  type character formula}
\label{sec:char}

In this section, we shall formulate and establish a Frobenius type
formula for the irreducible characters of the Hecke-Clifford
algebra.

\subsection{A  formula for ${\rm tr}(DT_{w_{(n)}})$}

%\begin{lemma}
%Suppose $1\leq k\leq n-1$ and $u=e_{i_1}\otimes\cdots\otimes e_{i_n}$ with $i_1,\ldots, i_n\in I(m|m)$.
%Then
 %\begin{align*}
 %\Psi_n(T_k)(u)|_u
 %=\left\{
 %\begin{array}{ll}
 %v, &\text{ if }i_k=i_{k+1}\geq 1, \\
 %-1, &\text{ if }i_k=i_{k+1}\leq -1,\\
 %v-1,&\text{ if }i_k<i_{k+1},\\
 %0,& \text{ otherwise}.
 %\end{array}
 %\right.
  %\end{align*}
%\end{lemma}
%\begin{proof}
%By Lemma~\ref{lem:Ol}, we have
%\begin{align}
%\Psi_n(T_k)(u)
%=&\bar{S}^{k,k+1}(e_{i_1}\otimes\cdots\otimes e_{i_n})\notag\\
%=&e_{i_1}\otimes\cdots\otimes e_{i_{k-1}}\otimes\bar{S}(e_{i_k}\otimes
%e_{i_{k+1}})\otimes e_{i_{k+2}}\otimes\cdots\otimes e_{i_n})\label{eq:Tkact}\\
%\end{align}
%The following formula will be useful for our computation later.

Recall that  $\bar{S}=QS\in\End(V^{\otimes 2})$.
\begin{lemma}  \label{lem:Sbar}
The following formula holds
for $k, \ell\in I(m|m)$:
\begin{align*}
&\bar{S}(e_k\otimes e_{\ell})= \\
&\left\{
\ds\begin{array}{llllllll}
ve_{\ell}\otimes e_k+(v-1)e_{-k}\otimes e_{-\ell},&\text{ if }k=\ell\geq 1,\\
-e_{\ell}\otimes e_k, &\text{ if }k=\ell\leq -1,\\
e_{\ell}\otimes e_k, &\text{ if }k=-\ell\geq 1,\\
ve_{\ell}\otimes e_k+(v-1)e_k\otimes e_{\ell},&\text{ if }k=-\ell\leq -1,\\
v^{\frac 12}e_{\ell}\otimes e_k+(v-1)e_{-k}\otimes e_{-\ell}+(v-1)e_k\otimes e_{\ell},
 &\text{ if } |k|<|\ell|\text{ and }\ell\geq 1,\\
v^{\frac 12}{\rm sgn}(k)e_{\ell}\otimes e_k, &\text{ if } |k|<|\ell|
 \text{ and }\ell\leq -1,\\
v^{\frac 12}e_{\ell}\otimes e_k+{\rm sgn}(\ell)(v-1)e_{-k}\otimes e_{-\ell},
 &\text{ if }|\ell|<|k|\text{ and }k\geq 1,\\
{\rm sgn}(\ell) e_{\ell}\otimes e_k+(v-1)e_k\otimes e_{\ell},&\text{ if }|\ell|<|k|
 \text{ and }k\leq -1.
\end{array}
\right.
\end{align*}
\end{lemma}

\begin{proof}
By \eqref{eq:operatorS}, we compute that
\begin{align}  \label{eq:actionS}
&v^{-\hf} S(e_k\otimes e_{\ell}) \\
&=  \sum_{i\leq j\in I(m|m)} (S_{ij}\otimes E_{ij})(e_k\otimes e_{\ell})
\notag \\
&= \sum_{i\leq j\in I(m|m)}(-1)^{|E_{ij}|\cdot |e_k|}S_{ij}(e_k)\otimes E_{ij}(e_{\ell})\notag\\
&= \sum^{\ell}_{i=-m}(-1)^{|E_{i\ell }|\cdot |e_k|}S_{i\ell }(e_k) \otimes e_i\notag\\
&=\left\{
\begin{array}{ll}
 S_{\ell\ell}(e_k)\otimes e_{\ell}+\sum_{1\leq i<\ell}S_{i\ell }(e_k)\otimes e_i
+\sum_{-m\leq i\leq -1}(-1)^{|e_k|}S_{i\ell}(e_k)\otimes e_i,
 &\text{ if }\ell\geq 1,\\
S_{\ell\ell}(e_k)\otimes e_{\ell}+\sum_{-m\leq i<\ell}S_{i\ell }(e_k)\otimes e_i,
 &\text{ if }\ell\leq -1.
\end{array}
\right.  \notag
\end{align}
Then the lemma is proved case-by-case using the definition of
$S_{ij}$ given in~\eqref{eq:S1}-\eqref{eq:S5}. Let us illustrate by
checking in detail the case when $|k|<|\ell|, \ell\geq 1$. In this
case, we have either $1\leq k<\ell \leq m$ or $-\ell<k\leq -1$. If
$1\leq k<\ell \leq m$, then it follows by \eqref{eq:S3} and
\eqref{eq:S5} that
\begin{align*}
v^\hf  \sum_{1\leq i<\ell}S_{i\ell }(e_k)\otimes e_i
&=(v-1)e_{\ell}\otimes e_k,
 \\
v^\hf \sum_{-m\leq i\leq -1}(-1)^{|e_k|}S_{i\ell}(e_k)\otimes e_i
&=-(v-1)e_{-\ell}\otimes e_{-k},
\end{align*}
and hence by \eqref{eq:actionS} we obtain that
$$
S(e_k\otimes e_{\ell})=
v^\hf e_k\otimes e_{\ell}+(v-1)e_{\ell}\otimes e_k-(v-1)e_{-\ell}\otimes e_{-k}.
$$
Therefore,
\begin{align*}
\bar{S}(e_{k}\otimes e_{\ell})
=&Q \big(v^\hf e_k\otimes e_{\ell}+(v-1)e_{\ell}\otimes e_k-(v-1)e_{-\ell}\otimes e_{-k}
 \big)\\
=&v^\hf e_{\ell}\otimes e_k +(v-1)e_k\otimes e_{\ell}+(v-1)e_{-k}\otimes e_{-\ell}.
\end{align*}
If $-\ell<k\leq -1$, then by~\eqref{eq:S3} and \eqref{eq:S5} we have
\begin{align*}
v^\hf \sum_{1\leq i<\ell}S_{i\ell }(e_k)\otimes e_i
&=(v-1)e_{-\ell}\otimes e_{-k},
  \\
v^\hf \sum_{-m\leq i\leq -1}(-1)^{|e_k|}S_{i\ell}(e_k)\otimes e_i
&=(v-1)e_{\ell}\otimes e_k,
\end{align*}
and hence by~\eqref{eq:actionS} we have
$$
S(e_k\otimes e_{\ell})=
v^\hf e_k\otimes e_{\ell}+(v-1)e_{-\ell}\otimes e_{-k}+(v-1)e_{\ell}\otimes e_k.
$$
Therefore,
$$
\bar{S}(e_{k}\otimes e_{\ell})=
v^\hf e_{\ell}\otimes e_k+(v-1)e_{-k}\otimes e_{-\ell}+(v-1)e_k\otimes e_{\ell}.
$$
The remaining cases can be verified similarly, and we skip the detail.
\end{proof}

For a composition $\ga=(\ga_1,\ldots,\ga_{\ell})$ of $n$, let
\begin{align}
T_{{\ga,j}}&= T_{\ga_1+\ldots+\ga_{j-1}+1}
T_{\ga_1+\ldots+\ga_{j-1}+2} \cdots T_{\ga_1+\cdots+\ga_j-1},
 \quad 1\leq j\leq \ell,
   \notag \\
T_{w_{\ga}}&=T_{{\ga,1}}T_{{\ga,2}}\cdots T_{{\ga,{\ell}}}.
  \label{eq:Twga}
\end{align}
Equivalently, $T_{w_{\ga}}$ is the Hecke algebra element corresponding to the permutation
\begin{equation}  \label{eq:wga}
w_{\ga}=(1,\ldots,\ga_1)(\ga_1+1,\ldots, \ga_1+\ga_2)
\cdots(\ga_1+\cdots+\ga_{\ell-1}+1,\ldots, \ga_1+\cdots+\ga_{\ell}).
\end{equation}
%We further introduce a short-hand notation $\xn$ for $T_{w_{(n)}}$, that is,
%$$
%\xn =T_1T_2\ldots T_{n-1}.
%$$
For $\uni=(i_1,i_2,\ldots,i_n)\in I(m|m)^n$ satisfying $ i_1\leq
i_2\leq\cdots\leq i_n$,  denote
\begin{align*}
f(\uni)&=\sharp\{1\leq k\leq n | i_k=i_{k+1}\geq 1\},\\
g(\uni)&=\sharp\{1\leq k\leq n | i_k=i_{k+1}\leq -1\},\\
h(\uni)&=\sharp\{1\leq k\leq n | i_k<i_{k+1}\}.
\end{align*}
We denote by $T(u)|_u$ the coefficient of $u$ in the linear
expansion of $T(u)$ in terms of the basis
$\{e_{i_1}\otimes\cdots\otimes e_{i_n} \mid i_1,\ldots, i_n\in
I(m|m)\}$ of $V^{\otimes n}$, for a linear operator $T\in\End_\K
(V^{\otimes n})$ and a basis element $u$.

\begin{lemma}\label{lem:trace1}
The trace of the operator $DT_{w_{(n)}}$
on $V^{\otimes n}$ is given by
$$
{\rm tr}(DT_{w_{(n)}})
=\sum_{\uni}v^{f(\uni)}(-1)^{g(\uni)}(v-1)^{h(\uni)}x_{|i_1|}\cdots x_{|i_n|},
$$
where the summation is over the $n$-tuples $\uni=(i_1,i_2,\ldots,i_n)\in I(m|m)^n$
which satisfy $i_1\leq i_2\leq\cdots\leq i_n$.
\end{lemma}

\begin{proof}
It suffices to show that, for $u=e_{i_1}\otimes e_{i_2}\otimes\cdots\otimes e_{i_n}$,
\begin{align*}
(DT_{w_{(n)}} u)|_u
=\left\{
\begin{array}{ll}
v^{f(\uni)}(-1)^{g(\uni)}v^{h(\uni)}x_{|i_1|}\cdots x_{|i_n|},
&\text{ if } i_1\leq i_2\leq\cdots\leq i_n,\\
0, &\text{ otherwise}.
\end{array}
\right.
\end{align*}
By definition,  we have $T_{w_{(n)}} =T_{w_{(n-1)}} T_{n-1}$.
It follows by Proposition~\ref{prop:Ol} and Lemma~ \ref{lem:Sbar} that
\begin{align*}
&(DT_{w_{(n)}} u)|_u\\
=&x_{|i_1|}\cdots x_{|i_n|}\cdot (T_{w_{(n)}} u)|_u\\
=&x_{|i_1|}\cdots x_{|i_n|}\cdot T_{w_{(n-1)}}
\big(e_{i_1}\otimes\cdots\otimes e_{i_{n-2}}\otimes
 \bar{S}  (e_{i_{n-1}}\otimes e_{i_n}) \big)|_u\\
=&\left\{
\begin{array}{llll}
v  x_{|i_1|}\cdots x_{|i_n|}  T_{w_{(n-1)}}(e_{i_1}\otimes\cdots\otimes
e_{i_{n-1}})|_{e_{i_1}\otimes\cdots\otimes e_{i_{n-1}}},
&\text{ if }i_{n-1}=i_n\geq 1,\\
-  x_{|i_1|}\cdots x_{|i_n|}  T_{w_{(n-1)}}(e_{i_1}\otimes\cdots\otimes
 e_{i_{n-1}})|_{e_{i_1}\otimes\cdots\otimes e_{i_{n-1}}},
&\text{ if }i_{n-1}=i_n\leq -1,\\
(v-1) x_{|i_1|}\cdots x_{|i_n|}
 T_{w_{(n-1)}}(e_{i_1}\otimes\cdots\otimes
 e_{i_{n-1}})|_{e_{i_1}\otimes\cdots\otimes e_{i_{n-1}}},
&\text{ if }i_{n-1}<i_n,\\
0,&\text{ otherwise}.
\end{array}
\right.
\end{align*}
Now the lemma follows by induction on $n$.
\end{proof}

For $s \ge 1$ and a composition $\alpha=(\alpha_1, \alpha_2, \ldots,\alpha_{\ell})$, set
\begin{align*}
\Delta_0&=1, \\
 \Delta_s&=v^{s-1}+(-1)^{s-1}+(v-1)\sum_{t=1}^{s-1}v^{t-1}(-1)^{s-t-1}
 =\frac{2(v^s-(-1)^s)}{v+1},
  \\
 \Delta_{\alpha}&=\Delta_{\alpha_1}\Delta_{\alpha_2}\cdots\Delta_{\alpha_{\ell}}.
\end{align*}
Let $m_\mu$ denote the monomial symmetric function associated to a partition $\mu$.

\begin{proposition} \label{prop:trace}
The trace of the operator $DT_{w_{(n)}}$
on $V^{\otimes n}$ is given by
$$
{\rm tr}(DT_{w_{(n)}})=\sum_{\mu\in\mc{P}_n}
\Delta_{\mu}(v-1)^{\ell(\mu)-1}m_{\mu}(x_1,\ldots,x_m).
$$
\end{proposition}

\begin{proof}
Given  $\alpha=(\alpha_1,\ldots,\alpha_m)\in\mc{CP}_n$, denote by
$\Ga(\alpha)$ the set consisting of $\uni=(i_1,\ldots, i_n)\in
I(m|m)^n$ satisfying $i_1\leq\cdots\leq i_n$ and
$\sharp\{j|~|i_j|=k\}=\alpha_k$ for $1\leq k\leq m$. By
Lemma~\ref{lem:trace1}, the coefficient  of the monomial
$x^{\alpha}:=x_1^{\alpha_1}\cdots x_m^{\alpha_m}$ in ${\rm
tr}(DT_{w_{(n)}})$, denoted by ${\rm
tr}(DT_{w_{(n)}})|_{x^{\alpha}}$,  is given by
$$
{\rm tr}(DT_{w_{(n)}})|_{x^{\alpha}}
=\sum_{\uni\in\Ga(\alpha)}v^{f(\uni)}(-1)^{g(\uni)}(v-1)^{h(\uni)}.
$$
Now for $\uni=(i_1,\ldots, i_n)$, we let
$$
\beta_k^+=\{j|i_j=k\},~ \beta_k^-=\{j| i_j=-k\}, \qquad 1\leq k\leq
m.
$$
Then each $\uni$ in  $\Ga(\alpha)$ is uniquely determined by the
pair $(\beta^+,\beta^-)\in\Z_+^m\times\Z_+^m$, where
$\beta^{\pm}=(\beta_1^{\pm}, \ldots,\beta_m^{\pm})$  satisfies
$\beta_k^++\beta_k^-=\alpha_k$ for $1\leq k\leq m$. In terms of
these notations, we rewrite
\begin{align*}
v^{f(\uni)} =&\prod_{1\leq k\leq m, \beta_k^+\geq 1}v^{\beta_k^+-1},\\
(-1)^{g(\uni)} =&\prod_{1\leq k\leq m, \beta_k^-\geq 1}(-1)^{\beta_k^--1},\\
(v-1)^{h(\uni)} =&(v-1)^{N(\beta^\pm)-1},
\end{align*}
where we have denoted
$$
N(\beta^\pm) =\sharp \{1 \le k \le m \mid \beta_k^+ >0\} +\sharp \{1
\le k \le m \mid \beta_k^- >0\}.
$$
Let $\Omega(\alpha)$ denote the set which consists of all the pairs
$(\beta^+,\beta^-)$ satisfying $\beta_k^++\beta_k^-=\alpha_k$, for
$1\leq k\leq m$, and let  $\ell(\alpha)$ denote the number of
nonzero parts in $\alpha$. Then
\begin{align*}
  {\rm tr} & (DT_{w_{(n)}})|_{x^{\alpha}}  \\
=&\sum_{\uni\in\Ga(\alpha)}v^{f(\uni)}(-1)^{g(\uni)}(v-1)^{h(\uni)}\\
=&\sum_{(\beta^+,\beta^-)\in\Omega(\alpha)}
 \prod_{1\leq k\leq m, \beta_k^+\geq 1}v^{\beta_k^+-1} \cdot
 \prod_{1\leq k\leq m, \beta_k^-\geq 1}(-1)^{\beta_k^--1} \cdot
(v-1)^{N(\beta^\pm)-1}\\
=&(v-1)^{-1}\sum_{(\beta^+,\beta^-)\in\Omega(\alpha)}\prod_{k=1}^m
v^{\beta_k^+-1+\delta_{\beta_k^+,0}}
\cdot(-1)^{\beta_k^- -1+\delta_{\beta_k^-,0}}
\cdot(v-1)^{2-\delta_{\beta_k^+,0}-\delta_{\beta_k^-,0}}\\
=&(v-1)^{\ell(\alpha)-1}\prod_{1\leq k\leq m,  \alpha_k> 0}
\Big(v^{\alpha_k-1}+(-1)^{\alpha_k-1}
+\sum_{j=1}^{\alpha_k-1}v^{j-1}(-1)^{\alpha_k-j-1}(v-1)\Big)\\
=&(v-1)^{\ell(\alpha)-1} \Delta_{\alpha}.
\end{align*}
Therefore,
\begin{align*}
{\rm tr}(DT_{w_{(n)}})
=&\sum_{\alpha=(\alpha_1,\ldots,\alpha_m)\in\Z_+^m,
\sum_k\alpha_k=n} (v-1)^{\ell(\alpha)-1} \Delta_{\alpha}
x_1^{\alpha_1}\cdots x_m^{\alpha_m}\\
=&\sum_{\mu\in\mc{P}_n}(v-1)^{\ell(\mu)-1}\Delta_{\mu}m_{\mu}(x_1,\ldots, x_m).
\end{align*}
The proposition is proved.
\end{proof}

For $n\geq 0$ and $x =(x_1,\ldots, x_m)$, we define $g_n(x;v)$ by
the following generating function in a variable $t$:
\begin{align}\label{eq:generating}
\sum_{n\geq 0}g_n(x;v)t^n=\prod_{i}
\frac{1 -tx_i}{1 +tx_i}\cdot \frac{1 +v tx_i}{1 -v tx_i},
\end{align}
and then set
\begin{equation}\label{eq:onecharacter}
\widetilde{g}_n(x;v)=\frac{1}{v-1}g_n(x;v).
\end{equation}

Proposition~\ref{prop:trace} can be reformulated as follows.

\begin{proposition}\label{prop:trace2}
For $n\geq 1$, we have
$
\text{tr}(DT_{w_{(n)}})=\widetilde{g}_n(x;v).
$
\end{proposition}

\begin{proof}
By Proposition~\ref{prop:trace}, we have
\begin{align*}
(v-1)^{-1} &+ \sum_{n\geq 1}\text{tr}(DT_{w_{(n)}})t^n
 \\
=&\sum_{n}t^n\sum_{\mu\in\mc{P}_n}(v-1)^{\ell(\mu)-1}\Delta_{\mu}m_{\mu}(x_1,\ldots, x_m)\\
=&\frac{1}{v-1}\prod_i \Big(1+\sum_{s\geq 1}(v-1)\Delta_sx_i^st^s\Big)\\
=&\frac{1}{v-1}\prod_i \left(1+\sum_{s\geq 1}(v-1)
\frac{2v^sx_i^st^s-2(-1)^sx_i^st^s}{v+1}\right)\\
=&\frac{1}{v-1}\prod_i \left(1+\frac{2(v-1)}{v+1}
\Big(\frac{vx_it}{1-vx_it}-\frac{-x_it}{1+x_it}\Big)\right)\\
=&\frac{1}{v-1}\prod_{i}\frac{1 -x_it}{1 +x_it}\cdot \frac{1 +v
x_it}{1 -v x_it}.
\end{align*}
This implies (and is indeed equivalent to) the proposition by using
\eqref{eq:generating} and \eqref{eq:onecharacter}.
\end{proof}

\begin{remark}
The symmetric function $g_n(x;v)$ also appears as  a special case
of the spin Hall-Littlewood functions (i.e., the one
associated to the one-row partition $(n)$) introduced in \cite{WW2}.
\end{remark}

\subsection{A Frobenius formula for characters of $\HC$}

\begin{lemma}\label{lem:producttrace}
Assume that
$\ga=(\ga_1,\ga_2,\ldots,\ga_{\ell})$ is a composition of $n$. Then
$$
{\rm tr}(DT_{w_{\ga}})
=\prod_{1\leq j \leq\ell}{\rm tr}(DT_{{\ga,j}}).
$$
\end{lemma}

\begin{proof}
Observe that
$$
T_{w_{\ga}}=\prod_{1\leq j\leq\ell}T_{{\ga,j}},
$$
and $T_{{\ga,j}}$'s commute with each other. By
Proposition~\ref{prop:Ol}, we note that $T_{\ga,1}$ acts only on the
first $\ga_1$ factors, $T_{\ga,2}$ acts only on the subsequent
$\ga_2$ factors  and so on. The lemma follows.
\end{proof}

For a partition $\mu=(\mu_1,\ldots,\mu_{\ell})$, we define
\begin{align*}
\widetilde{g}_{\mu}(x;v)=\widetilde{g}_{\mu_1}(x;v)\widetilde{g}_{\mu_2}(x;v)\cdots
\widetilde{g}_{\mu_{\ell}}(x;v).
\end{align*}
We are ready to establish a Frobenius type formula for the
characters $\zeta^\la$ of the Hecke-Clifford algebra $\HCK$.

\begin{theorem}\label{thm:Frobenius}
The following holds for each partition $\mu$ of $n$:
\begin{equation}\label{eq:Frobenius}
\widetilde{g}_{\mu}(x;v)
=\sum_{\la\in\mc{SP}_n}2^{-\frac{\ell(\la)+\delta(\la)}{2}}
Q_{\la}(x)\zeta^{\la}(T_{w_{\mu}}).
\end{equation}
\end{theorem}

\begin{proof}
By Lemma~\ref{lem:producttrace} and Proposition~\ref{prop:trace2}, one deduces that
\begin{align*}
{\rm tr}(DT_{w_{\mu}})
=&{\rm tr}(DT_{\mu,1})\cdots{\rm tr}(DT_{\mu,\ell})\\
=&\widetilde{g}_{\mu_1}(x;v) \cdots \widetilde{g}_{\mu_{\ell}}(x;v)
=\widetilde{g}_{\mu}(x;v).
\end{align*}
This together with Proposition~\ref{prop:generaltrace} implies the theorem.
\end{proof}

\begin{remark}
Recall the specialization at $q=1$ of the Frobenius character
formula for type A Hecke algebra established in \cite{Ram} recovers
the original formula of Frobenius \cite{Fr} for the irreducible
characters of symmetric groups.
For $r \ge 1$, we have
\begin{eqnarray*}
\widetilde{g}_r(x;v)|_{v=1} =
\left\{
 \begin{array}{ll}
 2p_r(x), & \text{ for } r \text{ odd} \\
 0,  & \text{ for } r \text{ even},
 \end{array}
 \right.
\end{eqnarray*}
where $p_r$ denotes the $r$th power sum symmetric function. Hence,
the Frobenius formula in Theorem~\ref{thm:Frobenius} specializes
when $v=1$ to a character formula for $\mf H^c_n$, which is
essentially equivalent to Schur's original character formula for the
spin symmetric groups \cite{Sch} (cf.~ \cite{Jo} and \cite{WW3}).
\end{remark}
\section{Trace functions on Hecke-Clifford algebra}
\label{sec:trace}

In this section, we will exhibit an explicit basis for the space of
trace functions on the Hecke-Clifford algebra.

\subsection{The trace functions}

Let $R$ be a commutative ring in which $2$ is invertible. For an
$R$-superalgebra $\mc{H}$ which is a free $R$-module, a {\em trace
function} on $\mc{H}$ is an $R$-linear map $\phi: \mc{H}\rightarrow
R$ satisfying
$$
\phi(hh')=\phi(h'h) \text{ for } h, h'\in\mc{H} ,\qquad  \phi(h)=0
\text{ for }h\in\od{\mc{H}}.
$$
For $h, h'\in\mc H$, define their commutator by $[h,h']=hh'-h'h$,
and let $[\mc H,\mc H]$ be the $R$-submodule of $\mc H$ spanned by
all commutators (not super-commutators!). Observe that a linear map
$\phi:\mc H\rightarrow R$ with $\phi(\od{\mc H})=0$ is a trace
function if and only if $\ev{[\mc H,\mc H]}\subseteq {\rm Ker}\phi$.
Thus, the space of trace functions on $\mc H$ is canonically
isomorphic to the dual space $\Hom_{R}(\ev{(\mc H/[\mc H,\mc H])},
R)$ of $\ev{(\mc H/[\mc H,\mc H])}=\ev{\mc H}/\ev{[\mc H,\mc H]}$.

Let ${\bf A}:=\Z[\frac{1}{2}][v,v^{-1}]\subseteq\C(v)$, and denote
by $\HCa$ the $\bf A$-subalgebra of $\HC$ generated by $T_1,\ldots,
T_{n-1}$ and $c_1,\ldots c_n$. Note that $\HCa$ is $\bf A$-free and
$\HC=\C(v^{\frac12})\otimes_{\bf A}\HCa$.

The main result of this subsection is Theorem~\ref{thm:spacetrace},
the proof of which requires a sequence of lemmas. Recall $T_{w_\ga}$
from \eqref{eq:Twga}.

\begin{lemma}\label{lem:Ram5.1}
For each $I\subseteq [n]$ and $\sigma\in S_n$, there exists an $\bf
A$-linear combination of the form
$$
\Ga_{I,\sigma}=\sum_{J\subseteq[n],
\ga\in\mc{CP}_n}a_{(I,\sigma),(J,\ga)}T_{w_{\ga}}C_J,
$$
where $\ell(w_\ga) \leq \ell(\sigma)$ and
$a_{(I,\sigma),(J,\ga)}\in\bf A$, such that
$$
T_{\sigma}C_I\equiv \Ga_{I,\sigma}\mod [\HCa, \HCa].
$$
\end{lemma}
\begin{proof}
Let $i$ be the smallest integer such that $\sigma(i)>i+1$.
We shall use a double induction involving
an induction on $\sigma(i)$ and a reverse induction on $i$.
Observe that if there does not exist such $i$, this can be regarded as the case
$i=n$ and
$\sigma$ must be equal to $w_\ga$ defined in \eqref{eq:wga}
for some composition $\ga=(\ga_1,\ga_2,\ldots,\ga_{\ell})$.

Let $j=\sigma(i)-1$. Since $\sigma(\sigma^{-1}(j))=j=\sigma(i)-1>i$,
the choice of $i$ implies that $\sigma^{-1}(j)>i$. This together
with $\sigma^{-1}(j+1)=i$ implies that
$\sigma^{-1}(j)>\sigma^{-1}(j+1)$ and hence $\ell(\sigma^{-1}
s_j)<\ell(\sigma^{-1})$, or equivalently,
$\ell(s_j\sigma)<\ell(\sigma)$. Then
$$
T_{\sigma}=T_{s_j}T_{s_j\sigma}.
$$
The following holds by the defining relation of $\HC$:
\begin{equation}\label{eq:CITj}
C_IT_{s_j}=T_{s_j}C_{I'}+\sum_{J\subseteq[n]}a_JC_J
\end{equation}
with $a_J\in\bf A$ and $I'\subseteq[n]$. We now consider two
cases separately.

(i) First assume that $\ell(s_j\sigma s_j)>\ell( s_j\sigma)$.
Then $T_{s_j\sigma s_j}=T_{ s_j\sigma}T_{s_j}$, and hence
\begin{align*}
T_{\sigma}C_I
&=T_{s_j}T_{ s_j\sigma}C_I\\
&\equiv T_{ s_j\sigma}C_IT_{s_j} \mod [\HCa, \HCa]
 \\
&=T_{ s_j\sigma}T_{s_j}C_{I'}+\sum_{J}a_JT_{ s_j\sigma}C_J
 \quad \text{ by } \eqref{eq:CITj} \\
&= T_{s_j\sigma s_j}C_{I'}+\sum_{J}a_J T_{ s_j\sigma}C_J.
\end{align*}
%and hence
%$$
%T_{\sigma}C_I\equiv T_{s_j\sigma s_j}C_{I'})+\sum_{J}a_J T_{ s_j\sigma}C_J  \mod [\HCa, \HCa].
%$$

(ii) Assume that $\ell(s_j\sigma s_j)<\ell( s_j\sigma)$.  In this
case, we have $T_{ s_j\sigma}=T_{s_j\sigma s_j}T_{s_j}$, and thus
$T_{\sigma}=T_{s_j}T_{s_j\sigma }=T_{s_j}T_{s_j\sigma s_j}T_{s_j}$.
Hence,
\begin{align*}
T_{\sigma}C_I
&=T_{s_j}T_{s_j\sigma s_j}T_{s_j}C_I\\
&\equiv T_{s_j\sigma s_j}T_{s_j}C_IT_{s_j} \mod [\HCa, \HCa].
\end{align*}
Again by \eqref{eq:CITj} we compute that
\begin{align*}
T_{s_j\sigma s_j}T_{s_j}C_IT_{s_j}
&=T_{s_j\sigma s_j}T_{s_j}T_{s_j}C_{I'}
+\sum_{J}a_JT_{s_j\sigma s_j}T_{s_j}C_J  \\
&= T_{s_j\sigma s_j}T^2_{s_j}C_{I'}
+\sum_{J}a_JT_{s_j\sigma s_j}T_{s_j}C_J\\
&= (v-1)T_{s_j\sigma s_j}T_{s_j}C_{I'}
+vT_{s_j\sigma s_j}C_{I'}+\sum_{J\subseteq[n]}a_JT_{s_j\sigma s_j}T_{s_j}C_J\\
&= (v-1)T_{ s_j\sigma}C_{I'}+vT_{s_j
\sigma s_j}C_{I'}+\sum_{J\subseteq[n]}a_JT_{ s_j\sigma}C_J.
\end{align*}
Thus, if $\ell(s_j\sigma s_j)<\ell( s_j\sigma)$, we have
$$
T_{\sigma}C_I\equiv  (v-1)T_{ s_j\sigma}C_{I'}+vT_{s_j
\sigma s_j}C_{I'}+\sum_{J\subseteq[n]}a_JT_{ s_j\sigma}C_J\mod [\HCa, \HCa].
$$

Observe that in each case, the resulted permutations
$\sigma':= s_j\sigma$ and $\sigma'':=s_j\sigma s_j$ satisfy
$\sigma'(k)=\sigma(k)=\sigma''(k)$ for $1\leq k\leq i-1<j-1$,
since  $\sigma(k)\leq k+1< j$ due to the choice of $i$.
In addition,
\begin{align*}
\sigma'(i)&=s_j(\sigma(i))=s_j(j+1)=j=\sigma(i)-1,\\
\sigma''(i)&=s_j(\sigma s_j(i))=s_j(\sigma(i))=s_j(j+1)=j=\sigma(i)-1.
\end{align*}
Note that $\ell(\sigma')=\ell(\sigma)-1$ in both cases. Moreover,
$\ell(\sigma'')=\ell(\sigma')+1=\ell(\sigma)$ in case (i), while
$\ell(\sigma'')<\ell(\sigma')<\ell(\sigma)$ in case (ii).
This completes the induction step, and hence the lemma is proved.
\end{proof}

Recall $\xn =T_1T_2\ldots T_{n-1}$. Denote by
$$
T_i' :=T_i -v+1 =vT_i^{-1}.
$$
%Recall that for an (ordered) subset $I =\{i_1, i_2, \ldots, i_k\} \subseteq [n]
%=\{1, \ldots, n\}$,  we denote $C_I =c_{i_1}c_{i_2} \ldots c_{i_k}$.
%Denote by
%$[C_J]f$ the coefficient of $CI$ in the expansion of $f \in \Cl_n$
%in the basis $\{C_I\}$.
%Denote by $a..n$ the subset $\{a, a+1,
%\ldots, n\}$ of $[n]$.

%\begin{lemma}  \label{lem:interval}
%\begin{enumerate}
%\item $\xn c_i =c_{i+1} \xn \in \HCa_n$ for $1\le i \le n-1$.

%\item $[C_I] (\xn . C_I) =0$ unless $I =a..n$ for $1\le a \le n$ or
%$I =\emptyset$.

%\item $\xn . c_n = {\bf x}_{n-1}.c_{n-1} + (v-1) v^{n-2} c_n$.

%\item $\xn . c_n =c_1 +(v-1) \big( c_2 +vc_3 +v^2 c_4 +\ldots +
%v^{n-2} c_n\big)$.

%\end{enumerate}
%\end{lemma}

%\begin{proof}
%(1) follows directly from the commutation relations in $\HCa_n$. (2)
%follows from (1).
%\end{proof}

%\begin{lemma}  \label{lem:coxelt}
%The following holds for the basic spin $\HCa_n$-module.
%\begin{enumerate}
%\item $\xn . c_n = {\bf x}_{n-1}.c_{n-1} + (v-1) v^{n-2} c_n$.

%\item $\xn . c_n =c_1 +(v-1) \big( c_2 +vc_3 +v^2 c_4 +\ldots +
%v^{n-2} c_n\big)$.

%\item $[C_{a..n}] (\xn.C_{a..n}) = (-1)^{n-a} (v-1) v^{a-2}, \quad 2 \le
%a \le n.$

%\item $[C_{1..n}] (\xn. C_{1..n}) =(-1)^{n-1}; \quad [1] (\xn.1)
%=v^{n-1}.$
%\end{enumerate}
%\end{lemma}

%\begin{proof}
%(1) follows from $T_{n-1} c_n =c_{n-1}T_{n-1} -(v-1) (c_{n-1}
%-c_n)$.

%(2) follows from (1) by induction on $n$.

%(3) and (4) follow from (2) and Lemma~\ref{lem:interval}(1).
%\end{proof}

%The following is a strengthening of Lemma~\ref{lem:coxelt}~(1)(2).
\begin{lemma}  \label{lem:coxHn}
The following identities hold
in $\HCa$:
\begin{align*}
 \xn c_i &=c_{i+1} \xn \quad  \text{  for } 1\le i \le n-1.
  \\
\xn  c_n &=c_1 T_1'T_2'\ldots T_{n-1}'
  +(v-1)  \left( c_2  T_{w_{(1)}} T_2'\ldots T_{n-1}'
+c_3T_{w_{(2)}} T_3'\ldots T_{n-1}' \right.
  \\
  & \left.
\quad  \qquad \quad\qquad \quad\qquad \quad\qquad \quad+\ldots +
c_{n-1} T_{w_{(n-2)}} T_{n-1}' +c_nT_{w_{(n-1)}}  \right).
\end{align*}
\end{lemma}

\begin{proof}
The first identity follows directly from the defining relations in $\HCa$.
Since $T_{n-1}c_n=c_{n-1}T_{n-1}-(v-1)(c_{n-1}-c_n)$, we obtain that
$$
\xn  c_n = T_{w_{(n-1)}} c_{n-1}T_{n-1}' + (v-1) c_nT_{w_{(n-1)}}.
$$
Now the second identity follows from this identity by induction on $n$.
\end{proof}
\begin{lemma}\label{lem:xnCI}
For $I \subseteq[n]$ with $|I|$ even, the following holds:
$$
\xn C_I\equiv\pm \xn \mod \ev{ [\HCa, \HCa]}.
$$
(The precise signs here and in the subsequent similar expressions
shall not be needed.)
\end{lemma}
\begin{proof}
Set $I=\{i_1,\ldots,i_k\}$.
Since $k=|I|$ is even, we have $i_1\leq n-1$ and hence by Lemma \ref{lem:coxHn}
$$
\xn C_I=\xn c_{i_1}\cdots c_{i_k}=c_{i_1+1}\xn c_{i_2}\cdots c_{i_k}.
$$
Therefore,
\begin{align*}
\xn C_I
&\equiv \xn c_{i_2}\cdots c_{i_k}c_{i_1+1} \mod \ev{ [\HCa, \HCa]}\\
& = \left\{
\begin{array}{cc}
(-1)^{k-1}\xn c_{i_1+1}c_{i_2}\cdots c_{i_k} ,&\text{ if }i_1+1<i_2,\\
(-1)^{k-2}\xn c_{i_3}\cdots c_{i_k} ,&\text{ if }i_1+1=i_2.
\end{array}
\right.
\end{align*}
In this way, we reduce $\xn C_I$ to a similar expression with smaller $|I|$
or with an increased $i_1$.
By repeating the above procedure, the lemma is proved.
\end{proof}

\begin{lemma}\label{lem:length2}
Let $\ga=(\ga_1,\ga_2)$ be a composition of $n$ with
$\ga_1,\ga_2>0$. Suppose
$I_1=\{i_1,\ldots,i_a\} \subseteq\{1,\ldots,\ga_1\},
I_2=\{j_1,\ldots, j_b\} \subseteq\{\ga_1+1,\ldots,\ga_1+\ga_2\}$
such that $a+b$ is even.
Then
$$
T_{w_{\ga}}C_{I_1}C_{I_2}
\equiv
\left\{
\begin{array}{ll}
0\qquad \mod \ev{ [\HCa, \HCa]},&\text{ if } a \text{ and } b \text{ are odd},\\
\pm T_{w_{\ga}} \mod \ev{ [\HCa, \HCa]}, &\text{ if } a \text{ and }
b \text{ are even}.
\end{array}
\right.
$$
\end{lemma}

\begin{proof}
Since $a+b$ is even,  $a$ and $b$ have the same parity.
Note that
\begin{align*}
T_{w_{\ga}}=T_{\ga,1} & T_{\ga,2}  =T_{\ga,2}T_{\ga,1},\\
T_{\ga,1}C_{I_2}=C_{I_2}T_{\ga,1},  & \qquad
T_{\ga,2}C_{I_1}=C_{I_1}T_{\ga,2}.
\end{align*}

If both $a$ and $b$ are odd, then
$
C_{I_1}C_{I_2}=-C_{I_2}C_{I_1}.
$
It follows from these identities above that
\begin{align*}
T_{w_{\ga}}C_{I_1}C_{I_2}
%&=T_{\ga,1}T_{\ga,2}C_{I_1}C_{I_2}\\
&=T_{\ga,1}C_{I_1}T_{\ga,2}C_{I_2}\\
&\equiv T_{\ga,2}C_{I_2}T_{\ga,1}C_{I_1}
%&=T_{\ga,2}T_{\ga,1}C_{I_2}C_{I_1} \\
%&=-T_{\ga,2}T_{\ga,1}C_{I_1}C_{I_2} \\
= -T_{w_\ga} C_{I_1}C_{I_2} ,
\end{align*}
where the notation $\equiv$ here and in similar expressions below is always
understood$\text{ mod}$
$\ev{ [\HCa, \HCa]}$.
Therefore,
$
T_{w_{\ga}}C_{I_1}C_{I_2}\equiv 0 \mod \ev{ [\HCa, \HCa]}.
$

If both $a$ and $b$ are even, then
$1\leq i_1\leq \ga_1-1$
and
$ \ga_1+1\leq j_1\leq \ga_1+\ga_2-1=n-1.
$
By a similar analysis as in Lemma~\ref{lem:xnCI}, one can obtain that
\begin{align*}
T_{w_{\ga}}C_{I_1}C_{I_2}
&=T_{\ga,1}T_{\ga,2}c_{i_1}c_{i_2}\cdots c_{i_a}c_{j_1}c_{j_2}\cdots c_{j_b}\\
&=-T_{\ga,1}c_{i_1}T_{\ga,2}c_{j_1}c_{i_2}\cdots c_{i_a}\cdots c_{j_2}c_{j_b}\\
&=-c_{i_1+1}c_{j_1+1}T_{\ga,1}T_{\ga,2}c_{i_2}\cdots c_{i_a}\cdots c_{j_2}c_{j_b}.
\end{align*}
Hence,
\begin{align*}
T_{w_{\ga}}C_{I_1}C_{I_2}
&=-c_{i_1+1}c_{j_1+1}T_{\ga,1}T_{\ga,2}c_{i_2}\cdots c_{i_a} c_{j_2}\cdots c_{j_b}\\
&\equiv-T_{\ga,1}T_{\ga,2}c_{i_2}\cdots c_{i_a}c_{j_2}\cdots c_{j_b}c_{i_1+1}c_{j_1+1} \\
&=\left\{
\begin{array}{cc}
T_{\ga,1}T_{\ga,2}c_{i_1+1}c_{i_2}
\cdots c_{i_a}c_{j_1+1}c_{j_2}\cdots c_{j_b} ,&\text{ if }i_1+1<i_2, j_1+1<j_2,\\
-T_{\ga,1}T_{\ga,2}c_{i_1+1}c_{i_2}
\cdots c_{i_a}c_{j_3}\cdots c_{j_b} ,&\text{ if }i_1+1<i_2, j_1+1=j_2,\\
-T_{\ga,1}T_{\ga,2}c_{i_3}\cdots c_{i_a}c_{j_1+1}c_{j_2}
\cdots c_{j_b},&\text{ if }i_1+1=i_2, j_1+1<j_2,\\
T_{\ga,1}T_{\ga,2}c_{i_3}\cdots c_{i_a}c_{j_3} \cdots
c_{j_b} ,&\text{ if }i_1+1=i_2, j_1+1=j_2.
\end{array}
\right.
\end{align*}
In this way, we reduce $T_{w_{\ga}}C_{I_1}C_{I_2}$ to a similar expression
with smaller $|I_1| +|I_2|$ or with increased $i_1$ and $j_1$.
Repeating the procedure, the proposition is proved.
\end{proof}

Given integers $a<b$, we shall denote by $[a..b]$ the set of
integers $k$ such that $a\le k \le b$. The following is a
generalization of Lemmas~\ref{lem:xnCI} and \ref{lem:length2}.

\begin{lemma}\label{lem:lengthgeneral}
Let $\ga=(\ga_1,\ga_2,\ldots,\ga_{\ell})$ be a composition of $n$,
and let $I \subseteq[n]$ with $|I|$ even. Let $I_k=I\cap
[(\ga_1+\ldots+\ga_{k-1}+1)..(\ga_1+\ldots+\ga_k)]$ for $1\leq
k\leq\ell$. Then
$$
T_{w_{\ga}}C_I\equiv\left\{
\begin{array}{ll}
\pm T_{w_{\ga}} \mod \ev{ [\HCa, \HCa]},
&\text{if every }|I_k|\text{ is even for }1\leq k\leq\ell,\\
0 \qquad \mod \ev{ [\HCa, \HCa]},&\text{otherwise}.
\end{array}
\right.
$$
\end{lemma}

\begin{proof}
The lemma for $\ell=1$ reduces to Lemma~\ref{lem:xnCI}. So assume
$\ell \ge 2$. If every $|I_k|$ is even for $1\leq k\leq\ell$, then
the lemma follows by a similar proof as in Lemma~\ref{lem:length2}
(which is the special case when $\ell=2$). Otherwise, suppose there
exists $1\leq a\leq\ell$ such that $|I_a|$ is odd. Without loss of
generality and for the sake of simplifying notations, we assume
$a=1$. Let $b>1$ be the smallest integer such that $I_b$ is odd
(such $b$ exists since $|I|=k$ is even). Therefore,
\begin{align*}
T_{w_{\ga}} C_{I}
&=(T_{\ga,1}C_{I_1}T_{\ga,2}C_{I_2}\cdots
T_{\ga,b-1}C_{I_{b-1}})(T_{\ga,b}C_{I_b})(T_{\ga,b+1}C_{I_{b+1}}
\cdots T_{\ga,\ell}C_{I_{\ell}})\\
&=(T_{\ga,1}C_{I_1})(T_{\ga,b}C_{I_b})(T_{\ga,2}C_{I_2}\cdots
T_{\ga,b-1}C_{I_{b-1}})(T_{\ga,b+1}C_{I_{b+1}}
\cdots T_{\ga,\ell}C_{I_{\ell}})\\
&=-(T_{\ga,b}C_{I_b})(T_{\ga,1}C_{I_1})(T_{\ga,2}C_{I_2}\cdots
T_{\ga,b-1}C_{I_{b-1}})(T_{\ga,b+1}C_{I_{b+1}}
\cdots T_{\ga,\ell}C_{I_{\ell}})\\
&\equiv-(T_{\ga,1}C_{I_1})(T_{\ga,2}C_{I_2}\cdots
T_{\ga,b-1}C_{I_{b-1}})(T_{\ga,b+1}C_{I_{b+1}}
\cdots T_{\ga,\ell}C_{I_{\ell}}) (T_{\ga,b}C_{I_b})  \\
& =-T_{w_{\ga}}C_{I}.
\end{align*}
Hence,
$
T_{w_{\ga}}C_{I}\equiv 0\mod { [\HCa, \HCa]}.
$
This completes the proof of the lemma.
\end{proof}

Denote by $\mathcal{H}_{n,\bf A}$ the subalgebra of $\HCa$ generated
by $T_1,\ldots, T_{n-1}$, which is the Hecke algebra over $\bf A$
associated to the symmetric group $S_n$.

\begin{lemma}\label{lem:GP8.2}
Suppose $\ga=(\ga_1,\ldots,\ga_{\ell})$ is a composition of $n$ and
let $\mu=(\mu_1,\ldots,\mu_{\ell})$ be the  partition obtained by a
rearrangement of the parts of $\ga$. Then
$$
T_{w_{\ga}}\equiv T_{w_{\mu}} \ev{\mod  [\HCa, \HCa]}.
$$
\end{lemma}

\begin{proof}
It is known (see \cite[Theorem 5.1]{Ram} or \cite[Section 8.2]{GP2})
that $ T_{w_{\ga}} \equiv T_{w_{\mu}}\mod [\mathcal{H}_{n,\bf
A},\mathcal{H}_{n,\bf A}]. $ The lemma now follows since $
[\mathcal{H}_{n,\bf A},\mathcal{H}_{n,\bf A}]\subseteq \ev{[\HCa,
\HCa]}. $
\end{proof}

Recall that $\mc{OP}_n$ denotes the set of odd partitions of $n$.

\begin{theorem}\label{thm:spacetrace}
For each $\sigma\in S_n$ and $I\subseteq [n]$ with $|I|$ even, there
exist $ f_{\sigma,I;\nu}\in \bf A$ such that
\begin{align}\label{eq:commutator}
T_{\sigma}C_I\equiv \sum_{\nu\in\mc{OP}_n}
f_{\sigma,I;\nu}T_{w_{\nu}} \mod \ev{ [\HCa, \HCa]}.
\end{align}
\end{theorem}

\begin{proof}
By Lemma~\ref{lem:Ram5.1}, Lemma~\ref{lem:lengthgeneral} and
Lemma~\ref{lem:GP8.2}, it suffices to show that, for  each partition
$\mu$ of $n$, there exists $f_{\mu;\nu}\in\bf A$ such that
\begin{equation}  \label{eq:munu}
T_{w_{\mu}}\equiv \sum_{\nu\in\mc{OP}_n}f_{\mu;\nu}T_{w_{\nu}} \mod \ev{ [\HCa, \HCa]}.
\end{equation}

Let us assume  for a moment  that $n$ is even. Set $y_n:=
c_1c_2\ldots c_n$. By  Lemma~\ref{lem:coxHn}, we calculate that
$$y_n^{-1} \xn y_n
=(c_1c_2\ldots c_n)^{-1} (c_2c_3\ldots c_n)\xn c_n
=-c_1 \xn c_n.
$$
By the second identity in Lemma~\ref{lem:coxHn} (and expanding the
$T_i'$ therein as a sum of monomials), $-c_1\xn c_n$ can be written
as $-\xn +$ a linear combination of elements of the form $c_kc_m
T_\sigma$ with $\ell(\sigma) \le n-2 =\ell(w_{(n)})-1$. Equivalently,
$y_n^{-1} \xn y_n=-c_1\xn c_n$ can be written as $-\xn +$ a linear
combination of elements of the form $T_\sigma c_ic_j$ with
$\ell(\sigma)<\ell(w_{(n)})$.

Now we come to the proof of \eqref{eq:munu}. Let us assume that
$\mu$ is a partition of $n$ with an even part $\mu_a$ for some $a\in
\{1,\ldots,\ell(\mu)\}$. Set
$$
y :=c_{(\mu_1+\ldots +\mu_{a-1}+1)}c_{(\mu_1+\ldots
+\mu_{a-1}+2)}\ldots c_{(\mu_1+\ldots +\mu_{a-1}+\mu_{a})}.
$$
Then, the computation in the previous paragraph is applicable to
$y^{-1} T_{\mu,a} y$, which in turn implies that
\begin{equation}\label{eq:oddminimal}
y^{-1} T_{w_\mu} y =-T_{w_\mu} +Z,
\end{equation}
where $Z$ is a linear combination of elements of the form $T_\sigma
c_ic_j$ where $\ell(\sigma) < \ell(w_\mu)$. By
Lemma~\ref{lem:Ram5.1}, $T_\sigma c_ic_j$ is a linear combination of
$T_{w_\la} C_J$ with $\ell(w_\la) \le \ell(\sigma)$. Hence, $Z$ is a
linear combination of elements of the form $T_{w_\la} C_I$ with
$\ell(w_\la)< \ell(w_\mu)$. By Lemma~\ref{lem:lengthgeneral},
$$
\hf Z \equiv  \text{  a linear combination of $T_{w_\la}$ with
$\ell(w_\la)< \ell(w_\mu)$,$\mod \ev{ [\HCa, \HCa]}$.}
$$
On the other hand, we have
$
T_{w_\mu} =(T_{w_\mu} y)y^{-1} \equiv y^{-1} T_{w_\mu} y\mod \ev{ [\HCa, \HCa]}.
$
This together with~\eqref{eq:oddminimal} implies that
$$
T_{w_\mu} \equiv\frac12  Z \mod \ev{ [\HCa, \HCa]}.
$$
Now the proof is completed by induction on the length $\ell(w_\mu)$.
\end{proof}

\subsection{The space of trace functions and character table of $\HCa$}

Theorem~\ref{thm:spacetrace} has the following implication.

\begin{theorem}\label{th:basistrace}
$\ev{(\HCa/[\HCa,\HCa])}$ is a free $\bf A$-module, with a basis
consisting of the images of $T_{w_{\nu}}$ for $\nu\in\mc{OP}_n$
under the projection $\HCa\rightarrow\HCa/[\HCa,\HCa]$.
\end{theorem}

\begin{proof}
By Theorem~\ref{thm:spacetrace}, $\ev{(\HCa/[\HCa,\HCa])}$ is
spanned by the images of the elements $T_{w_{\nu}}$ with
$\nu\in\mc{OP}_n$ under the projection
$\HCa\rightarrow\HCa/[\HCa,\HCa]$. Passing to the splitting field
$\mathbb K$ for Hecke-Clifford algebra, the images of the elements
$T_{w_{\nu}}$ with $\nu\in\mc{OP}_n$ remain to be a spanning set for
$\ev{(\HCK/[\HCK,\HCK])}$. By Proposition~\ref{prop:JN}, $\HCK$ is
semisimple and its non-isomorphic irreducible characters are
parametrized by $\mc{SP}_n$. It follows that the dimension of the
space of trace functions on $\HCK$ is
$$
\dim_{\mathbb K} \ev{(\HCK/[\HCK,\HCK])} = |\mc{SP}_n|=|\mc{OP}_n|.
$$
Hence the images of $T_{w_{\nu}}$ with $\nu\in\mc{OP}_n$ are
linearly independent in $\ev{(\HCK/[\HCK,\HCK])}$ as well as in
$\ev{(\HCa/[\HCa,\HCa])}$. This proves the theorem.
\end{proof}

\begin{corollary}\label{cor:spacetrace}
Every trace function $\phi:\HCa\rightarrow {\bf A}$ is uniquely
determined by its values on the elements $T_{w_{\nu}}$ for
$\nu\in\mc{OP}_n$. Moreover, the polynomials $f_{\sigma,I;\nu}$ in
\eqref{eq:commutator} are uniquely determined by $\sigma, I$ and
$\nu$.
\end{corollary}

For $\sigma\in S_n$ and $I\subseteq [n]$ with $|I|$ even and
$\nu\in\mc{OP}_n$, $f_{\sigma,I;\nu}$ are called {\em class
polynomials}, and they are spin analogues of the class polynomials
introduced by Geck-Pfeiffer \cite[Definition 1.2(2)]{GP1} (cf.
\cite[Section 8.2]{GP2}) for Hecke algebras associated to finite
Weyl groups.
The square matrix
$$
\big[\zeta^{\la}(T_{w_{\nu}}) \big]_{\la\in\mc{SP}_n,\nu\in\mc{OP}_n}
$$
is called the {\em character table} of the Hecke-Clifford algebra
$\HCK$. By Corollary~\ref{cor:spacetrace} and the linear
independence of irreducible characters $\zeta^{\la}$ for
$\la\in\mc{SP}_n$, the square matrix $[\zeta^{\la}(T_{w_{\nu}})]
_{\la,\nu}$ is invertible in $\K$.

\begin{remark}
Note that $w_{\nu}$ is a minimal length representative in the
conjugacy class $C$ in $S_n$ of cycle type $\nu\in\mc{OP}_n$. Let
$w_C$ be another minimal length representative in the same conjugacy
class. It is known from \cite[Theorem 1.1]{GP1} that
$T_{w_{\nu}}\equiv T_{w_C}\mod [\mc{H}_{n,\bf A}, \mc{H}_{n,\bf A}]$
and hence $T_{w_{\nu}}\equiv T_{w_C}\mod [\HCa,\HCa]$ thanks to
$[\mc{H}_{n,\bf A}, \mc{H}_{n,\bf A}]\subseteq \ev{[\HCa,\HCa]}$.
Thus our definition of the character table of $\HCK$ is independent
of the choice of minimal length representatives in conjugacy classes
of cycle type being odd partitions of $n$. Moreover, specializing
$v=1$, the matrix $(\zeta^{\la}(T_{w_{\nu}}))
_{\la\in\mc{SP}_n,\nu\in\mc{OP}_n}$ reduces to  the character table
of the algebra $\mf H^c_n$  (cf. \cite{WW3}).
\end{remark}

For $\nu\in\mc{OP}_n$, define an $\bf A$-linear map $f_{\nu}:\HCa
\rightarrow {\bf A}$ by
\begin{align*}
f_\nu(T_{\sigma}C_I)
=
\left\{
\begin{array}{ll}
f_{\sigma,I;\nu}, & \text{ if }|I|\text{ is even}\\
0, &  \text{ if }|I|\text{ is odd}.
\end{array}
\right.
\end{align*}

\begin{proposition} \label{prop:HCpoly}
For each $\nu\in\mc{OP}_n$, $f_{\nu}$ is a trace function on $\HCa$
which satisfies
\begin{equation}  \label{eq:fdual}
f_{\nu}(T_{w_{\rho}})=\delta_{\nu,\rho}, \qquad \text{ for } \rho\in\mc{OP}_n.
\end{equation}
Moreover, $\{f_{\nu}|\nu\in\mc{OP}_n\}$ is a basis for the space of
trace functions on $\HCa$.
\end{proposition}
\begin{proof}
Recall that the distinct irreducible characters of
$\HCK=\K\otimes_{\bf A}\HCa$ are given by $\zeta^{\la}$ for
$\la\in\mc{SP}_n$. By Theorem~\ref{thm:spacetrace},  for any
$\la\in\mc{SP}_n, \sigma\in S_n$ and $I\subseteq [n]$ with $|I|$
even we have
$$
\zeta^{\la}(T_{\sigma}C_I)
=\sum_{\nu\in\mc{OP}_n} f_{\sigma, I;\nu}\zeta^{\la}( T_{w_{\nu}})
=\sum_{\nu\in\mc{OP}_n} f_{\nu}(T_{\sigma}C_I)\zeta^{\la}( T_{w_{\nu}}).
$$
Then by the invertibility of the character table
$(\zeta^{\la}(T_{w_{\nu}})) _{\la\in\mc{SP}_n,\nu\in\mc{OP}_n}$ we
can write
$$
f_{\nu}(T_{\sigma}C_I)=\sum_{\la\in\mc{SP}_n}g_{\la;\nu}\zeta^{\la}(T_{\sigma}C_I)
$$
for some $g_{\la;\nu}\in\K$. Therefore $f_{\nu}$ is a trace function
on $\HCK$ and hence a trace function on $\HCa$. Now \eqref{eq:fdual}
follows from the definition of $f_{\nu}$ and \eqref{eq:commutator}.
Then by Theorem~\ref{th:basistrace}, $\{f_{\nu}|\nu\in\mc{OP}_n\}$
forms a basis of the space of trace functions on $\HCa$.
\end{proof}
\section{Spin generic degrees for Hecke-Clifford algebra}
\label{sec:degree}

In this section, we shall introduce the spin generic degrees for the
Hecke-Clifford algebra and show that it  coincides with spin fake
degrees associated to the spin symmetric groups introduced in
\cite{WW1, WW3}.

\subsection{Basics on symmetric superalgebras}

Let $\mc{H}$ be an $R$-superalgebra which is free and of finite rank
over a commutative ring $R$ containing $\hf$. A trace function
$\phi: \mc{H}\rightarrow R$ is called a {\em symmetrizing trace
form} if the bilinear form
$$
\mc{H}\times\mc{H}\longrightarrow R,\quad  (h,h')\mapsto \phi(hh')
$$
is non-degenerate, i.e., there exists a homogeneous basis $\mc B$ of
$\mc H$ such that the determinant of matrix
$(\phi(b_1b_2))_{b_1,b_2\in\mc B}$ is a unit in $R$. In this case,
$(\mc{H}, \phi)$ or $\mc{H}$ is called a {\em symmetric
superalgebra}.

\begin{remark}\label{rem:matrixtrace}
Let $\mc{H}=M(V)$ or $\mc{H}=Q(V)$ over a field $\F$. Then every
trace function on $\mc{H}$ is a scalar multiple of the usual matrix
trace ${\rm tr}$. Note that $(\mc{H}, {\rm tr})$ is symmetric.
\end{remark}

In the remainder of this subsection, we assume that $\mc{H}$ is
symmetric with a symmetrizing trace form $\phi$, and describe some
basic results for $\mc H$. Though most are straightforward
superalgebra generalizations of the well-known classical results
(cf. \cite[Chapter~7]{GP2}), we need to make precise a possible
factor of $2$ due to type $\texttt Q$ simple $\mc H$-modules.

If $\mc{B}$ is a $\Z_2$-homogeneous basis for $\mc{H}$, we denote by
$\mc{B}^{\vee}=\{b^{\vee}|b\in\mc{B}\}$ the dual basis, which  is
also homogenous and satisfies that $\phi(b^{\vee}b')=\delta_{b,b'}$.
Suppose $V, V'$ are $\mc{H}$-modules. For any homogenous map
$f\in\Hom_{R}(V, V')$, we define $I(f)\in\Hom_{R}(V, V')$ by letting
$$
I(f)(v)=\sum_{b\in\mc{B}}(-1)^{|f||b|}b^{\vee}f(bv), \qquad \text{ for }v\in V.
$$
It follows by essentially the same proof as for \cite[Lemma
7.1.10]{GP2} with appropriate superalgebra signs inserted that
$I(f)$ is independent of the choice of the homogeneous basis
$\mc{B}$, and moreover $I(f)\in\Hom_{\mc{H}}(V,V')$.

Let $\F$ be a filed of characteristic not equal to 2.
From now on, we assume that $\mc H$ is a finite dimensional
superalgebra over a field $\F$ with a symmetrizing trace $\phi$. The
following lemma is the superalgebra analogue of
\cite[Theorem~7.2.1]{GP2}, which can be proved in the same way.

\begin{lemma}
  \label{lem:splitcV}
Let $V$ be a split irreducible $\mc H$-module. Then there exists a
unique element $c_V\in\F$ such that
$$
I(f)=c_V{\rm tr}(f)~{\rm id}_V, \qquad \text{ for } f\in\ev{\End_{\F}(V)}.
$$
\end{lemma}

The element $c_V$ is called the {\em Schur element} of $V$.
Let us compute the Schur element of the unique irreducible representation of
the simple superalgebras over $\F$.

\begin{example}\label{ex:example}
(1) Let $\mc H=Q(V)$ with $V=\F^{m|m}$.
Clearly $V$ is an irreducible
$\mc H$-module of type $\texttt Q$.
Let $v_1,\ldots, v_m$ be a basis
of $\ev V$ and $v_{-1},\ldots, v_{-m}$ be a basis of $\od V$,
and let $J\in\End_{\F}(V)$ be the automorphism sending $v_k$ to
$v_{-k}$ for $1\leq k\leq m$. Then $\mc H$ consists of $2m\times 2m$ matrices
of the form:
\begin{equation*}
\begin{pmatrix}
a&b\\
-b&a\\
\end{pmatrix},
\end{equation*}
where $a$ and $b$ are arbitrary $m\times m$ matrices.
Observe that
$\mc B=\{g_{ij}:=E_{i,j}+E_{-i,-j}|1\leq i,j\leq
m\}\cup\{h_{ij}:=E_{-i,j}-E_{i,-j}|1\leq i,j\leq m\}$ is a
basis of $\mc H$ and the dual basis with respect to the usual
matrix trace ${\rm tr}$ is $\mc B^{\vee}=\{g_{ij}^{\vee}=\frac
12g_{ji}|1\leq i,j\leq
m\}\cup\{h_{ij}^{\vee}=-\frac{1}{2}h_{ji}|1\leq i,j\leq m\}$.
%Take $f\in\ev{\End_{\F}(V)}$
%be such that $f(v_j)=\sum_{-m\leq k\leq m} f_{kj}v_k$ with $f_{kj}\in\F$.
Then a direct computation shows that, for $f\in\ev{\End_{\F}(V)}$,
$$
I(f)(v_k)=\frac{{\rm tr}(f)}{2}v_k, \qquad   \text{ for } k\in I(m|m).
$$
By Lemma~\ref{lem:splitcV},  the Schur element of the irreducible $\mc H$-module $V$
(with respect to the usual matrix trace) equals $\frac 12$.

(2) Let $\mc H=M(V)$ with $V=\F^{r|m}$.
Observe that $V$ is naturally an irreducible
$\mc H$-module of type $\texttt M$. A similar (and somewhat simpler)
calculation as in (1) shows  that  the Schur element of $V$ (with
respect to the usual matrix trace) equals 1.
%
%Let $v_1,\ldots, v_r$ be a basis of $\ev V$ and $v_{r+1},\ldots, v_{r+m}$
%a basis of $\od V$, and let $\mc B=\{E_{ij}|1\leq i, j\leq r+m\}$ be the corresponding basis of
%$\mc H$. Then the dual basis respect to the usual matrix trace ${\rm tr}$ is
%$\mc B^{\vee}=\{E_{ij}^{\vee}=E_{ji}|1\leq i,j\leq r+m\}$.
%A direct computation shows that, for $f\in\ev{\End_{\F}(V)}$,
%$$
%I(f)(v_k)={\rm tr}(f)v_k, \qquad \text{  for } 1\leq k\leq r+m.
%$$
%It follows from this and Lemma~\ref{lem:splitcV} that
%the Schur element of $V$
%(with respect to the usual matrix trace) equals 1.
\end{example}

We denote by ${\rm Irr}(\mc H)$ the complete set of non-isomorphic irreducible $\mc H$-modules.
Let $\chi_V$ denote the character of an irreducible $\mc H$-module $V$, and write
$$
\delta(V)=\left\{
\begin{array}{ll}
0,&\text{ if }V \text{ is of type }\texttt M,\\
1,&\text{ if }V \text{ is of type }\texttt Q.
\end{array}
\right.
$$

\begin{proposition}\label{prop:Schurelem}
Suppose that $\mc H$ is a split semisimple superalgebra over $\mathbb F$.
Then the Schur element $c_V$ for every irreducible $\mc H$-module $V$ is nonzero.
Moreover,
$$
\phi=\sum_{V\in{\rm Irr}(\mc H)}\frac{1}{2^{\delta(V)}c_V}\chi_V.
$$
\end{proposition}

\begin{proof}
Write $\mc H$ as a direct sum of simple superalgebras:
\begin{align}\label{eq:Wedderburn1}
\mc H=\bigoplus_{V\in{\rm Irr}(\mc H)}H(V).
\end{align}
Then the irreducible characters $\chi_V$ can be identified with the
usual matrix trace on $H(V)$. By Remark~\ref{rem:matrixtrace}, the
restriction of the trace form $\phi$ to $H(V)$ is a scalar multiple
of the irreducible character $\chi_V$ for each $V\in{\rm Irr}(\mc
H)$, i.e.,
$$
\phi=\sum_{V\in{\rm Irr}(\mc H)}d_V\cdot \chi_V
$$
for some scalar  $d_V\in\F$, which must be nonzero thanks to the
non-degeneracy of $\phi$. Let $\mc B=\cup_{V\in{\rm Irr}\mc H}\mc
B(V)$ be a homogeneous basis of $\mc H$ which is compatible with the
decomposition \eqref{eq:Wedderburn1} and let $\widetilde{\mc B}(V)$
be the basis in $\mc H(V)$ dual to $\mc B(V)$ with respect to the
trace function $\chi_V$ on $H(V)$. Then $\cup_{V\in{\rm Irr}(\mc
H)}\{d_V^{-1}b |b\in\widetilde{\mc B}(V)\}$ is the basis dual to
$\mc B$ with respect to the trace form $\phi$. Now fix an
irreducible $\mc H$-module $V$. For $f\in\ev{\End_{\F}(V)}$ and
$v\in V$, we have
\begin{align*}
c_V{\rm tr}(f) v =& I(f)(v)=\sum_{b\in\mc B}b^{\vee}f(bv) \\
=&\sum_{V'\in{\rm Irr}\mc H}\sum_{b\in\mc B(V')}b^{\vee}f(bv)
= \sum_{b\in\mc B(V)}b^{\vee}f(bv)\\
=&\frac{1}{d_V}\sum_{b\in\mc B(V)}\widetilde{b}f(bv)
= \frac{1}{d_V}\frac{1}{2^{\delta(V)}}{\rm tr}(f)(v),
\end{align*}
where the fourth equality is due to $bv=0$ for $b\in\mc B(V')$ with
$V'\neq V$ and the last equality follows from
Example~\ref{ex:example} and the fact that the summation on the
right hand side is the defining formula for the Schur element of $V$
with respect to the usual matrix trace $\chi_V$ on $H(V)$. Therefore
$ c_V=\frac{1}{2^{\delta(V)}d_V}, $ and the proposition follows.
\end{proof}

\begin{remark}
As in \cite[Corollary~7.2.4]{GP2}, the following orthogonality
relation between split simple characters $\chi_V$ and $\chi_{V'}$
holds for a symmetric superalgebra $\mc H$:
$$
\sum_{b\in\mc B} \chi_V(b)\chi_{V'}(b^{\vee})
=\left\{
\begin{array}{ll}
2^{\delta(V)}   c_V{\rm dim}V,&\text{  if  }\chi_V=\chi_{V'}, \\
0,&\text{otherwise}.
\end{array}
\right.
$$
\end{remark}

\subsection{The symmetrizing trace form $\gimel$ and Schur elements}

Define a trace function $\gimel: \HCa\rightarrow {\bf A}$ which is
characterized by the conditions
\begin{align*}
\gimel(T_{w_{\nu}})&= \Big(\frac{v-1}{2}\Big)^{n-\ell(\nu)},
 \qquad \text{ for }\nu\in\mc{OP}_n, \\
\gimel(z)&=0, \qquad \text{ for }z\in\od{(\HCa)}.
\end{align*}
By Theorem~\ref{th:basistrace} and Corollary~\ref{cor:spacetrace},
$\gimel$ is well-defined and unique. We still denote by $\gimel$ the
corresponding trace function on $\HCK$ by a base change. We shall
compute the Schur elements for $\HCK$ with respect to $\gimel$. We
first prepare some notations.

Given a partition $\la$, suppose that the main diagonal of the Young
diagram $\la$ contains $r$ cells. Let $\alpha_i=\la_i-i$ be the
number of cells in the $i$th row of $\la$ strictly to the right of
$(i,i)$, and let $\beta_i=\la_i'-i$ be the number of cells in the
$i$th column of $\la$ strictly below $(i,i)$, for $1\leq i\leq r$.
We have $\alpha_1>\alpha_2>\cdots>\alpha_r\geq0$ and
$\beta_1>\beta_2>\cdots>\beta_r\geq0$. Then the Frobenius notation
for a partition is
$\la=(\alpha_1,\ldots,\alpha_r|\beta_1,\ldots,\beta_r)$. For
example, if $\la=(5,4,3,1)$ whose corresponding Young diagram is
$$
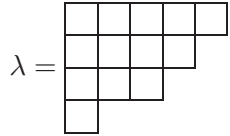
\la =\young(\,\,\,\,\,,\,\,\,\,,\,\,\,,\,)
$$
then $\alpha =(4,2,0),
\beta=(3,1,0)$ and hence $\la=(4,2,0|3,1,0)$ in Frobenius
notation.

Suppose that $\la$ is a strict partition of $n$. Let $\la^*$ be the
associated  {\em  shifted diagram}, that is,
$$
\la^*=\{(i,j)~|~1\leq i\leq l(\la), i\leq j\leq\la_i+i-1
\}
$$
which is obtained from the ordinary Young diagram by shifting the
$k$th row to the right by $k-1$ squares, for each $k$. Denoting
$\ell(\la)=\ell$, we define the {\em double partition}
$\widetilde{\la}$ to be $\widetilde{\la}=(\la_1,\ldots,\la_\ell|
\la_1-1,\la_2-1,\ldots,\la_\ell-1)$ in Frobenius notation. Clearly,
the shifted diagram $\la^*$ coincides with the part of
$\widetilde{\la}$ that lies strictly above the main diagonal. For each cell
$(i,j)\in \la^*$, denote by $h^*_{ij}$ the associated hook length in
the Young diagram $\widetilde{\la}$, and set the content
$c_{ij}=j-i$.

\begin{example}
Let $\la= (4, 3, 1)$. The corresponding shifted diagram $\la^*$ and
double diagram $\widetilde\la$ are
$$
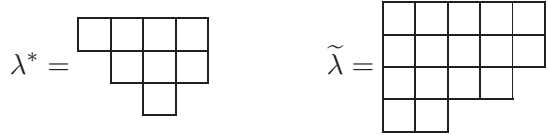
\la^*=\young(\,\,\,\,,:\,\,\,,::\,)
\qquad \qquad
\widetilde{\la}=\young(\,\,\,\,\,,\,\,\,\,\,,\,\,\,\,,\,\,)
$$
The contents of $\la$ are listed in the corresponding cell of $\la^*$ as follows:
$$
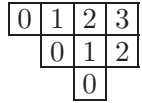
\young(0123,:012,::0)
$$
The shifted hook lengths for each cell in $\la^*$ are  the
usual hook lengths for the corresponding cell in $\la^*$, as
part of the double diagram $\widetilde \la$, as follows:
$$
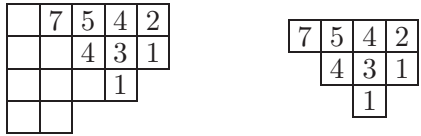
\young(\,7542,\,\,431,\,\,\,1,\,\,)
\qquad \qquad \young(7542,:431,::1)
$$
\end{example}

For  $\la\in\mc{SP}_n$, let
$Q_{\la}(v^{\bullet}):=Q_{\la}(1,v,v^2,\ldots)$ be the principal
specialization of Schur $Q$-function $Q_\la$ at $v^{\bullet}=
(1,v,v^2,\ldots).$ The following formula for $Q_{\la}(v^{\bullet})$
appeared as \cite[Theorem B]{WW1} (also see \cite{Ro} for a
different form).

\begin{proposition}
  \label{prop:SchurQ}
Suppose $\la\in\mc{SP}_n$. Then
$$
Q_\la(v^\bullet) =\frac{v^{n(\la)}\prod_{\Box\in
\la^*}(1+v^{c_{\Box}})}{\prod_{\Box\in \la^*}(1-v^{h^*_{\Box}})}.
$$
\end{proposition}

Now we compute the Schur elements for simple $\HCK$-modules.

\begin{theorem} \label{th:spinSchur}
$\gimel$ is a symmetrizing trace form on $\HCK$. For
$\la\in\mc{SP}_n$, the Schur element $c^{\la}$ of the simple
$\HCK$-module $U^{\la}$ with respect to $\gimel$ is given by
\begin{equation}\label{eq:elemSchur}
c^{\la}=2^{n+\frac{\ell(\la)-\delta(\la)}{2}} \frac{\prod_{\Box\in
\la^*}(1-v^{h^*_{\Box}})}{v^{n(\la)}(1-v)^n\prod_{\Box\in
\la^*}(1+v^{c_{\Box}})}.
\end{equation}
\end{theorem}
\begin{proof}
Set $u_{\la}$ to be the inverse of the right hand side of~\eqref{eq:elemSchur}.
%$$
%u_{\la}:=2^{-n-\frac{\ell(\la)+\delta(\la)}{2}}
%\frac{v^{n(\la)}(1-v)^n\prod_{\Box\in
%\la^*}(1+v^{c_{\Box}})}{\prod_{\Box\in \la^*}(1-v^{h^*_{\Box}})}.
%$$
Recall from Proposition~\ref{prop:JN} that $\HCK$ is semisimple. By
Proposition~\ref{prop:Schurelem},  in order to establish the
theorem, it suffice to show that
\begin{equation}  \label{eq:gimelzeta}
\gimel=\sum_{\la\in\mc{SP}_n}u_{\la}\zeta^{\la}.
\end{equation}

Recall the function $\widetilde{g}_{r}(x;v)$ from \eqref{eq:onecharacter}.
Specializing  \eqref{eq:generating} at  $x=v^{\bullet}$, we obtain that
$$
 \sum_{n \ge 0}
\tilde{g}_n(v^\bullet; v) t^n =\frac1{v-1} \cdot \frac{1-t}{1+t}
=\frac1{v-1} \big(1+\sum_{n \ge 1} 2 (-1)^n t^n \big).
$$
Hence we have
\begin{equation*}
\tilde{g}_n(v^\bullet; v)  =\frac{2 (-1)^n}{v-1}, \qquad n \ge 1,
\end{equation*}
and
\begin{equation}  \label{eq:special g}
\tilde{g}_\mu(v^\bullet; v) =\frac{2^{\ell(\mu)}
(-1)^n}{(v-1)^{\ell(\mu)}}, \text{ for }\mu\in\mc{P}_n.
\end{equation}
By the Frobenius formula in Theorem~\ref{thm:Frobenius} and the
definition of $\gimel$, we obtain that
\begin{align}  \label{eq:Qform}
\sum_{\la \in\mc{SP}_n}  2^{-\frac{\ell(\la)
+\delta(\la)}2}Q_\la(v^\bullet) \zeta^\la(T_{w_\nu})
=&\frac{2^{\ell(\nu)}(-1)^n}{(v-1)^{\ell(\nu)}}
=\frac{2^n}{(1-v)^n}\gimel(T_{w_{\nu}})
\end{align}
for all $\nu\in\mc{OP}_n$. Then by Corollary~\ref{cor:spacetrace},
one deduces that
$$
\gimel=\sum_{\la \in\mc{SP}_n}  2^{-n-\frac{\ell(\la) +\delta(\la)}2}(1-v)^nQ_\la(v^\bullet)
\zeta^\la.
$$
Now \eqref{eq:gimelzeta} follows from this identity and
Proposition~\ref{prop:SchurQ}. The theorem is proved.
\end{proof}
It follows from the definition of $\gimel$ that $\gimel(T_{w_{\nu}})=
\Big(\frac{v-1}{2}\Big)^{n-\ell(\nu)}$ for odd partition $\nu$ of
$n$. The following states that the formula actually
hold for all partitions of $n$.
\begin{corollary}
For all $\mu\in\mc P_n$, we have:
$$
\gimel(T_{w_{\mu}})=
\Big(\frac{v-1}{2}\Big)^{n-\ell(\mu)}.
$$
\end{corollary}
\begin{proof}
Let $\mu\in\mc P_n$.
By Theorem~\ref{th:spinSchur} (or equivalently,
\eqref{eq:gimelzeta}), we obtain
\begin{align*}
\gimel(T_{w_{\mu}})
&=\sum_{\la \in\mc{SP}_n}  2^{-n-\frac{\ell(\la) +\delta(\la)}2}(1-v)^nQ_\la(v^\bullet)
\zeta^\la(T_{w_{\mu}})\\
&=\big(\frac{1-v}{2}\big)^n
\sum_{\la \in\mc{SP}_n}  2^{-\frac{\ell(\la) +\delta(\la)}2}Q_\la(v^\bullet)
\zeta^\la(T_{w_{\mu}})\\
&=\big(\frac{1-v}{2}\big)^n\tilde{g}_\mu(v^\bullet; v)\\
&=\Big(\frac{v-1}{2}\Big)^{n-\ell(\mu)},
\end{align*}
where the last two equalities are due to Theorem~\ref{thm:Frobenius}
and \eqref{eq:special g}, respectively.
This proves the corollary.
\end{proof}
\subsection{The generic degrees for $\HCK$}

Denote by $P_{n}=\sum_{\sigma\in S_n}v^{\ell(\sigma)}$ the
Poincar{\'e} polynomial of the symmetric group $S_n$, and we can
formally regard $2^n P_n$ as the Poincar{\'e} polynomial of  $\mc
H^c_n$. It is known that the  Poincar{\'e} polynomial $P_{n}$ is
given by
$$
P_n=\frac{(1-v)(1-v^2)\cdots (1-v^n)}{(1-v)^n}.
$$
Define the {\em spin generic degree} $D^{\la} =D^\la (v)$ associated
to the irreducible $\HCK$-module $U^{\la}$, for $\la \in \mc{SP}_n$,
to be
$$
D^{\la}=\frac{2^nP_n}{c^{\la}}.
$$
The following is a reformulation of Theorem~\ref{th:spinSchur} by
definition of spin generic degrees.

\begin{theorem}   \label{th:genDeg}
The following formula for the spin generic degrees holds: for
$\la\in\mc{SP}_n$,
$$
D^{\la}=2^{-\frac{\ell(\la)-\delta(\la)}{2}}
\frac{v^{n(\la)}(1-v)1-v^2)\cdots(1-v^n)\prod_{\Box\in
\la^*}(1+v^{c_{\Box}})}{\prod_{\Box\in \la^*}(1-v^{h^*_{\Box}})}.
$$
\end{theorem}

\begin{remark}
Note that the specialization $\gimel$ at $v=1$ recovers the standard
symmetrizing trace form on $\mf H^c_n$, which is a twisted group
algebra of a double cover of the hyperoctahedral group. Moreover,
the specialization
$$
D^\la |_{v=1} = 2^{n-\frac{\ell(\la)-\delta(\la)}{2}}
\frac{n!}{\prod_{\Box\in \la^*} h^*_{\Box}}
$$
is the degree of the irreducible $\HC$-module $U^{\la}$. Our
definition of spin generic degrees for Hecke-Clifford algebras is 
analogous to Hecke algebras $\mc H_W$ associate to finite Weyl
groups $W$ (cf. \cite[Section 8.1.8]{GP2}). The canonical
symmetrizing trace form $\tau$ on $\mc H_W$ satisfies $\tau(1)=1$
and $\tau(T_{\sigma})=0$ for $1\neq \sigma\in W$.
\end{remark}

\begin{remark}
Though various connections between characters and generic degrees of
Hecke algebras have been explored in literature, our approach of
deriving the closed formula for $D^\la$ directly from the Frobenius
character formula is quite elegant and seems to be new even in the
usual Hecke algebra setting. We hope to apply the same strategy
elsewhere to revisit the generic degrees (or more general notion of
weights) for Hecke algebras.
\end{remark}

\subsection{Spin fake degrees for the symmetric group}
In this subsection, we shall take $\F=\C$.
The symmetric group $S_n$ acts on $\C^n$ and then on the symmetric
algebra $S^*\C^n$, which is identified with  $\C[x_1,\ldots, x_n]$
naturally. It is well known that the algebra of $S_n$-invariant on
$S^*\C^n$ is a polynomial algebra in the elementary symmetric
polynomials $e_1,\ldots, e_n$. The coinvariant algebra of $S_n$ is
defined to be
$$
(S^*\C^n)_{S_n}=S^*\C^n/I,
$$
where $I$ denotes the ideal generated by $e_1,\ldots, e_n$. By a
classical theorem of Chevalley the coinvariant algebra
$(S^*\C^n)_{S_n}$ is a graded regular representation of $S_n$.
Following Lusztig \cite{Lu},  the graded multiplicity of the Specht
modules $S^{\la}$ of $S_n$ in the coinvariant algebra is known as
the fake degree of $S^{\la}$, for $\la\in\mc P_n$ (cf. \cite[Section
5.3.3]{GP2}).

Note that the induced module ${\rm ind}^{\mf H^c_n}_{\C S_n}
(S^*\C^n)_{S_n}$ is a graded regular representation of $\mf H^c_n$.
Recall from \cite{WW3} that the {\em spin fake degree} of the
irreducible $\mf H^c_n$-module $U_1^{\la}$ with $\la\in\mc{SP}_n$ is
defined to be
$$
d^{\la}(t)=\sum_{j\geq 0}t^j \dim \Hom_{\mf H^c_n} \big(U_1^{\la},
{\rm ind}^{\mf H^c_n}_{\C S_n}(S^j\C^n)_{S_n} \big).
$$
The spin fake degrees have been computed in \cite[Theorem~A]{WW1}
(though the terminology was introduced later; see
\cite[Theorem~5.8]{WW3}). A comparison with Theorem~\ref{th:genDeg}
leads to the following.

\begin{corollary}
The spin generic degrees coincides with the spin fake degrees, that
is,
$$
D^{\la}(v)=d^{\la}(v), \qquad \text{ for all } \la\in\mc{SP}_n.
$$
\end{corollary}
This is parallel to the classical fact (due to Steinberg \cite{S},
cf. \cite{Lu, GP2}) that the generic degrees for the Hecke algebra
$\mc H_n$ associated to the symmetric group $S_n$ coincide with the
fake degrees for $S_n$, which is a type $A$ phenomenon.

\section{Trace functions on the spin Hecke algebra}
\label{sec:spinH}
\subsection{The spin Hecke algebra $\sH$}

Recall \cite{W} that the {\em spin Hecke algebra} $\sH$ is a
$\C(v^{\frac 12})$-superalgebra generated by the odd elements
$\rr_i, 1 \le i \le n-1$, subject to the following relations:
\begin{align}
\rr_i^2 &=
-(v^2 +1) \label{eq:rri2}\\
\rr_i \rr_j &= -\rr_j \rr_i  \quad (|i-j|>1)  \label{eq:rrij}\\
\rr_i \rr_{i+1} \rr_i - \rr_{i+1} \rr_i \rr_{i+1} &= (v-1)^2
(\rr_{i+1} -\rr_{i}).  \label{braidspin}
\end{align}

Set
\begin{align}
T_i^{\Phi} &:=
-\frac{1}{2}R_i(c_i-c_{i+1})+\frac{v-1}{2}(1-c_ic_{i+1}) \in
\sH\otimes\Cl_n,
  \label{eq:TiPhi}  \\
R^{\Psi}_i &:=(c_i-c_{i+1})T_i+(v-1)c_{i+1} \in \HC. \notag
\end{align}
The tensor superalgebra $\sH\otimes\Cl_n$ here is understood in the
sense of \eqref{eq:superotimes}.

\begin{proposition} \cite{W}
 \label{prop:PhiPsi}
There exist  isomorphisms $\Phi$ and $\Psi$ inverse to each other:
\begin{align*}
\Phi:\HC\longrightarrow \sH\otimes\Cl_n, & \quad
 \Psi:  \sH\otimes\Cl_n \longrightarrow \HC\\
 \Phi(T_i) = T_i^{\Phi}, &\quad
 \Phi(c_i) = c_i,
  \\
\Psi(R_i) = R^{\Psi}_i, & \quad \Psi(c_i) = c_i, \quad \text{for all
admissible } i.
\end{align*}
\end{proposition}

Set $\sHK=\K\otimes_{\C(v^{\frac 12})}\sH$. It is known that the
Clifford algebra $\Cl_n$ is a simple superalgebra with a unique
irreducible module $U_n$, which is of type $\texttt M$ if $n$ is
even and of type $\texttt Q$ if $n$ is odd. Moreover
$$
{\rm dim}U_n=\left\{
\begin{array}{ll}
2^k, &\text{ if } n=2k,\\
2^{k+1},&\text{ if } n=2k+1.
\end{array}
\right.
$$
Thanks to Lemma~\ref{tensorsmod}, Proposition~\ref{prop:JN} and the
above algebra isomorphisms, one sees that for each $\la\in\mc{SP}_n$
there exists an irreducible $\sHK$-module $U^{\la}_{-}$ with
character $\zeta^{\la}_-$ such that $\{U^{\la}_{-}\mid
\la\in\mc{SP}_n\}$ is a complete set of non-isomorphic irreducible
$\sHK$-modules, and moreover,
\begin{align}\label{eq:HCspin}
U^{\la} \cong
 \left\{
 \begin{array}{ll}
 2^{-1} U^{\la}_{-}\otimes U_n, & \text{ if } n \text{ is odd and }
 \ell(\la) \text{ is even},
 \\
  U^{\la}_{-}\otimes U_n, & \text{ otherwise.}
 \end{array}
 \right.
\end{align}

\subsection{The space $\ev{(\sHa/[\sHa,\sHa])}$}

We shall convert the study of the trace functions on the
Hecke-Clifford algebra $\HC$ in Section~\ref{sec:trace} to the spin
Hecke algebra $\sH$. Recall  ${\bf A} =\Z[\frac12] [v,v^{-1}].$
Denote by $\sHa$ the $\bf A$-subalgebra of the spin Hecke algebra
$\sH$ generated by $R_1,\ldots, R_{n-1}$. For $\sigma\in S_n$ with a
fixed arbitrary reduced expression
$\underline{\sigma}=s_{i_1}s_{i_2}\ldots$, we denote by $
R_{\underline{\sigma}} =R_{i_1}R_{i_2}\cdots. $ Then it follows from
\cite{W} that $\sHa$ is a free $\bf A$-module of rank $n!$ and
that $\{R_{\underline{\sigma}} \mid \sigma\in S_n\}$ is an $\bf
A$-basis of $\sHa$. Hence, the analogues of the isomorphisms $\Phi$
and $\Psi$ in Proposition~\ref{prop:PhiPsi} (which will be denoted
by the same notations) make sense over $\K$ or over $\bf A$.

\begin{lemma}  \label{lem:vanish}
Let $\underline{\sigma}$ be an arbitrary reduced expression of
$\sigma \in S_n$, and let $I\subseteq[n]$ be nonempty. Assume that
the element $R_{\underline\sigma}C_I$ is even. Then
$R_{\underline\sigma}C_I$ belongs to the commutator subspace $\ev
{[\sHa\otimes\Cl_n, \sHa\otimes\Cl_n]}$.
\end{lemma}

\begin{proof}
Denote $I=\{i_1, \ldots, i_k\}$. Since $R_{\underline{\sigma}}C_I$
is an even element, we have either (i) $R_{\underline{\sigma}}$ is
even and $k$ is even, or (ii) $R_{\underline{\sigma}}$ is odd and
$k$ is odd. In both cases, we have
\begin{align*}
R_{\underline{\sigma}}C_I
% &=R_{\underline{\sigma}}c_{i_1}\cdots c_{i_{k-1}}c_{i_k} \\
&=-c_{i_k}R_{\underline{\sigma}}c_{i_1}\cdots c_{i_{k-1}}
 \\
 &\equiv -(R_{\underline{\sigma}}c_{i_1}\cdots c_{i_{k-1}}) c_{i_k}\mod
[\sHa\otimes\Cl_n,\sHa\otimes\Cl_n] \\
 &=-R_{\underline{\sigma}}C_I.
 \end{align*}
Since $2$ is invertible in $\bf A$, the lemma is proved.
\end{proof}

\begin{lemma}  \label{lem:H-free}
The space $\ev{(\sHa/[\sHa,\sHa])}$ is a free $\bf A$-module.
\end{lemma}

\begin{proof}
By Lemma~\ref{lem:vanish}, the space $\ev{(\sHa\otimes\Cl_n/
[\sHa\otimes\Cl_n, \sHa\otimes\Cl_n])}$ is spanned by images of the elements
$R_{\underline{\sigma}}$ with $\ell(\sigma)$ being even under the projection
 $\sHa\otimes\Cl_n\rightarrow \sHa/[\sHa\otimes\Cl_n,\sHa\otimes\Cl_n]$.
So the natural map $\iota$
from $\ev{(\sHa/[\sHa,\sHa])}$ to
$\ev{(\sHa\otimes\Cl_n/[\sHa\otimes\Cl_n, \sHa\otimes\Cl_n])}$
(which is naturally identified via the isomorphism $\Phi$ with
$\ev{(\HCa/[\HCa,\HCa])}$) is surjective. By a base change, $\iota$
extends to a surjective map over $\K$. On the other hand, since
$\HCK$ and $\sHK$ are split semisimple with simple modules
parametrized by $\mc{SP}_n$, we have
\begin{equation}  \label{eq:dim}
\dim_\K \ev{(\sHK /[\sHK,\sHK])} =|\mc{SP}_n| =\dim_\K
\ev{(\HCK/[\HCK,\HCK])}.
\end{equation}
So the map $\iota$ is actually an isomorphism over $\K$ and hence
over $\bf A$. Now the lemma follows from the $\bf A$-freeness of
$\ev{(\HCa/[\HCa,\HCa])}$ by Theorem~\ref{th:basistrace}.
\end{proof}

For a composition $\ga\in\mc{CP}_n$ with $\ell(\ga)=\ell$, the
permutation $w_{\ga}$ (see \eqref{eq:wga}) has a unique reduced
expression given by
$$
\underline{w_{\ga}}=(s_1s_2\ldots s_{\ga_1-1})(s_{\ga_1+1}\ldots
s_{\ga_1+\ga_2-1})\cdots(s_{\ga_1+\ldots+\ga_{\ell-1}+1}\ldots
s_{n-1}).
$$
\begin{lemma}\label{lem:minimal}
Let $\ga$ be a composition of $n$ with $\ell(w_{\ga})$ being even and
$\mu$ be the corresponding partition of $\ga$. The following holds:
\begin{enumerate}
\item If $\mu\not\in\mathcal{OP}_n$, then $R_{\underline{w_{\ga}}} \equiv 0
\mod [\sHa,\sHa]_{\bar{0}}$.
\item If $\mu\in \mathcal{OP}_n$, then $R_{\underline{w_{\ga}}}
\equiv R_{\underline{w_{\mu}}}\mod [\sHa,\sHa]_{\bar{0}}$.
\end{enumerate}
\end{lemma}

\begin{proof}
Suppose $\mu\not\in\mathcal{OP}_n$ and
$\ga=(\ga_1,\ldots,\ga_{\ell})$. Let $a$ be the smallest integer
such that $\ga_a$ is even, and let $b$ be the smallest integer such
that $b>a$ and $\ga_b$ is even (which exists as $\ell(w_{\ga})$ is
even). Write
 $$
R_{\underline{w_{\ga}}}=R_{\ga,1}R_{\ga,2}\cdots R_{\ga,{\ell}},
 $$
where $R_{\ga,k}=R_{\ga_1+\ldots+\ga_{k-1}+1}\cdots
R_{\ga_1+\ldots+\ga_{k-1}+\ga_k-1}$ for $1\leq k\leq\ell$. Then
\begin{align*}
R_{\underline{w_{\ga}}}
 \equiv&
(R_{\ga,a}R_{\ga,{a+1}}\cdots R_{\ga,{b-1}}R_{\ga,b}R_{\ga,{b+1}}\cdots
R_{\ga,{\ell}})
R_{\ga,1}R_{\ga,2}\cdots R_{\ga,{a-1}}
 \\
=&R_{\ga,a}R_{\ga,b}R_{\ga,{a+1}}\cdots R_{\ga,{b-1}}R_{\ga,{b+1}}\cdots
R_{\ga,{\ell}}
R_{\ga,1}R_{\ga,2}\cdots R_{\ga,{a-1}}\\
=&-R_{\ga,b}R_{\ga,a}R_{\ga,{a+1}}\cdots R_{\ga,{b-1}}R_{\ga,{b+1}}\cdots
R_{\ga,{\ell}}
R_{\ga,1}R_{\ga,2}\cdots R_{\ga,{a-1}}\\
\equiv &- R_{\ga,a}R_{\ga,{a+1}}\cdots R_{\ga,{b-1}}R_{\ga,{b+1}}\cdots
R_{\ga,{\ell}}
R_{\ga,1}R_{\ga,2}\cdots R_{\ga,{a-1}}R_{\ga,b}  \\
=&-R_{\underline{w_{\ga}}},
\end{align*}
where $\equiv$ is understood$\mod [\sHa,\sHa]_{\bar{0}}$ here and
below. Therefore, $ R_{\underline{w_{\ga}}}\equiv 0.$

Now suppose $\mu\in\mathcal{OP}_n$. Using an argument similar to
Lemma~\ref{lem:producttrace}, one can obtain that $
\zeta^{\la}(R^{\Psi}_{\underline{w_{\ga}}})
=\zeta^{\la}(R^{\Psi}_{\underline{w_{\mu}}}) $ for every irreducible
character $\zeta^{\la}$ of $\HCK$. This implies that
\begin{equation*}
\zeta^{\la}_-(R_{\underline{w_{\ga}}})
=\zeta^{\la}_-(R_{\underline{w_{\mu}}}), \qquad \text{ for each }
{\la} \in \mc{SP}_n.
\end{equation*}
This together with Lemma~\ref{lem:H-free} implies that
$R_{\underline{w_{\ga}}}\equiv R_{\underline{w_{\mu}}}\mod
[\sHa,\sHa]_{\bar{0}}$.
\end{proof}

\begin{lemma}\label{lem:spinminimal}
Suppose that $w_C$ is a minimal length representative in the
conjugacy class $C$ of cycle type $\mu\in\mc P_n$.
 Then,
 \begin{align*}
R_{\underline{w_C}}\equiv\left\{
\begin{array}{ll}
\pm R_{\underline{w_\mu}}\mod [\sHa,\sHa],&\text{ if }\mu\in\mc{OP}_n,\\
0\quad\mod [\sHa,\sHa],&\text{ otherwise}.
\end{array}
\right.
\end{align*}
\end{lemma}

\begin{proof}
Suppose $\mu=(\mu_1,\ldots,\mu_{\ell})$ with $\ell(\mu)=\ell$.
The minimal element $w_C$ must be of the form
$$
w_{C}=(s_{i_1^1}s_{i_2^1}\cdots
s_{i^1_{\ga_1-1}})(s_{i^2_1}s_{i^2_2}\cdots
s_{i^2_{\ga_2-1}})\cdots(s_{i^{\ell}_1}s_{i^{\ell}_2}\cdots
s_{i^{\ell}_{\ga_{\ell}-1}}),
$$
where $\ga=(\ga_1,\ldots,\ga_{\ell})$ is a composition obtained by
rearranging the parts of $\mu$ and
$$
\{i^k_1,i^k_2,\ldots,i^k_{\ga_k-1}\}=\{\ga_1+\cdots+\ga_{k-1}+1,\ldots,
\ga_1+\cdots+\ga_{k-1}+\ga_k-1\}
$$
for $1\leq k\leq\ell$. Recall from \eqref{eq:wga} that $w_\ga$ is
the permutation associated to $\ga$.

{\bf Claim. } We have
$
R_{\underline{w_C}}\equiv \pm  R_{\underline{w_{\ga}}}\mod
[\sHa,\sHa]_{\bar{0}}.
$

Indeed, one reduces quickly the proof of the claim to the case
$\ga=(n)$. In this case, we write $w_C=s_{i_1}s_{i_2}\cdots
s_{i_{n-1}}$ with $i_a\neq i_b$ for $1\leq a\neq b\leq n-1$. If
$i_j=j$ for $1\leq j\leq n-1$, then $w_C=w_{\ga}$. Otherwise,
suppose $a$ is the smallest integer such that $i_a\neq a$. We shall
prove the claim for $\ga=(n)$ by reverse induction on $a$. Observe
that $i_a>a$, and hence
\begin{align*}
R_{\underline{w_C}} & =R_1R_2\cdots R_{a-1}R_{i_a}R_{i_{a+1}}\cdots
R_{i_{n-1}} \\
&=(-1)^{a-1}R_{i_a}R_1R_2\cdots R_{a-1}R_{i_{a+1}}\cdots R_{i_{n-1}}
\\
&\equiv (-1)^{a-1}R_1R_2\cdots R_{a-1}R_{i_{a+1}}\cdots
R_{i_{n-1}}R_{i_a} \mod [\sHa,\sHa].
\end{align*}
%For general $\nu$, this argument is also applicable. More precisely,
%suppose $w_C\neq w_{\nu}$. Let k, j be the smallest integers such
%that $i_j^k>\nu_1+\nu_2+...+\nu_{k-1}+j$. Then using a similar
%argument as above, we can move $R_{i_j^k}$ to the end.

If $i_{a+1}\neq a$ we can apply the above argument again to
$R_1R_2\cdots R_{a-1}R_{i_{a+1}}\cdots R_{i_{n-1}} R_{i_a}$ to move
$R_{i_{a+1}}$ to the end. By repeating the procedure, we obtain that
$$
R_{\underline{w_C}}
\equiv \pm R_{\underline{w'_C}} \mod [\sHa,\sHa],
$$
where $w'_C$ is a reduced expression of the form $w'_C=s_1s_2\cdots
s_a s_{i'_{a+1}}\cdots s_{i'_{n-1}}$. Then by induction assumption,
we have
$$
R_{\underline{w'_C}}\equiv \pm R_{\underline{w_{\ga}}}\mod
[\sHa,\sHa].
$$
Therefore the claim is proved. Now the lemma follows from
Lemma~\ref{lem:minimal}.
\end{proof}

\begin{theorem}\label{thm:spintrace}
Let $\sigma\in S_n$ with $\ell(\sigma)$ even and let
$\underline{\sigma}$ be a reduced expression of $\sigma$. Then there
exist $f^-_{{\underline{\sigma}},\nu} \in{\bf A}$ such that
$$
R_{\underline{\sigma}} \equiv \sum_{\nu \in \mathcal{OP}_n}
f^-_{{\underline{\sigma}},\nu} R_{\underline{w_\nu}} \mod [\sHa, \sHa]_{\bar{0}}.
$$
\end{theorem}

\begin{proof}
It is more flexible to use induction to establish the following.

{\bf Claim.} For $\sigma\in S_n$ with $\ell(\sigma)$ being even and
an arbitrary reduced expression $r(\sigma)$, there exist constants
$f^-_{r(\sigma),\ga}\in{\bf A}$ with $\ga\in\mc{CP}_n$ such that
$$
R_{r(\sigma)}=\sum_{\ga\in\mc{CP}_n, \ell(w_{\ga}) \text{ even
}}f^-_{r(\sigma),\ga}R_{\underline{w_{\ga}}} \mod [\sHa,
\sHa]_{\bar{0}}.
$$
Note that the theorem follows immediately by the claim,
Lemma~\ref{lem:minimal} and Lemma~\ref{lem:spinminimal}.

To prove the claim, we will follow an approach similar to the proof
of Lemma~ \ref{lem:Ram5.1} or \cite[Theorem~ 5.1]{Ram}. Let $i$ be
the smallest integer such that $\sigma(i)>i+1$. We shall use the
induction on $\ell(\sigma)$ and $\sigma(i)$, and reverse induction
on $i$. Note that  if there does not exist such $i$, this can be
regarded as the case $i=n$ and $\sigma$ much be of the form
$w_{\ga}$ for some $\ga\in\mc{CP}_n$. Thus, $\sigma$ has the unique
reduced expression $r(\sigma)=\underline{w_{\ga}}$ and the claim
follows.

Let $j=\sigma(i)-1$. Since $\sigma(\sigma^{-1}(j))=j=\sigma(i)-1>i$,
the choice of $i$ implies that $\sigma^{-1}(j)>i$. This together
with $\sigma^{-1}(j+1)=i$ implies that
$\sigma^{-1}(j)>\sigma^{-1}(j+1)$ and hence $\ell(\sigma^{-1}
s_j)<\ell(\sigma^{-1})$, or equivalently,
$\ell(s_j\sigma)<\ell(\sigma)$. Let $\sigma'=s_j\sigma$ and let
$r(\sigma')$ be a reduced expression of $\sigma'$. Then $r(\sigma)$
and $s_j r(\sigma')$ are two reduced expressions for $\sigma$. By
the defining relations among $R_i$, we have
$$
R_{r(\sigma)} =\pm R_jR_{r(\sigma')} + \sum_{\ell(w) < \ell(\sigma)} a_{\sigma,w}
R_{r(w)},
$$
where $a_{\sigma,w} \in {\bf A}$. By induction on $\ell(\sigma)$, we
may assume the claim holds for $R_{r(w)}$ for $w$ of length less
than $\sigma$. Hence we are reduced to show the claim holds for
$R_{s_jr(\sigma')}=R_jR_{r(\sigma')}$. Let $\sigma''=\sigma's_j$.

If $\ell(\sigma'')>\ell( \sigma')$, then $r(\sigma''):=r(\sigma')s_j$
is a reduced expression of $\sigma''$ and hence
\begin{align*}
R_jR_{r(\sigma')}
&\equiv R_{r(\sigma')}R_j \mod [\sHa, \sHa]\\
&= R_{r(\sigma'')}\mod [\sHa, \sHa].
\end{align*}
Then using the argument similar to the proof of
Lemma~\ref{lem:Ram5.1}, the induction on $i$ and reverse induction
on $\sigma(i)$ apply to $\sigma''$, and the claim follows.

Otherwise, assume $\ell(\sigma'')<\ell( \sigma')$. Fix a reduced
expression $r(\sigma'')$ for $\sigma''$. Then $r(\sigma')$ and
$r(\sigma'')s_j$ are two reduced expression for $\sigma'$ and again
by defining  relations among $R_i$, we have
$$
R_{r(\sigma')} =\pm R_{r(\sigma'')}R_j + \sum_{\ell(w) \leq \ell(\sigma'')} b_w R_{r(w)},
$$
where $b_{w} \in {\bf A}.$
Hence,
\begin{align*}
R_j R_{r(\sigma')} &\equiv R_{r(\sigma')} R_j  \mod [\sHa, \sHa]\\
&=\pm R_{r(\sigma'')}R_j^2 +\sum_{\ell(w) \leq \ell(\sigma'')} b_w R_{r(w)}R_j \mod [\sHa, \sHa].
\end{align*}
Since $\ell(r(w)s_j) \leq \ell(\sigma'')+1<\ell(\sigma)$, induction on the
length of $\sigma$ applies to the second summand, and the first
summand is also clear since $\ell(\sigma'') <\ell(\sigma)$ and $R_j^2
=-(1+v^2)$.

This completes the proof of the claim and hence the theorem.
\end{proof}

\begin{corollary}\label{cor:spinspacetrace}
$\ev{(\sHa/[\sHa,\sHa])}$ is an $\bf A$-free module, with a basis
consisting of the images of $R_{\underline{w_{\nu}}}$ for
$\nu\in\mc{OP}_n$ under the projection $\sHa\rightarrow
\sHa/[\sHa,\sHa]$. Every trace function on $\sHa$ is uniquely
determined by its values on $R_{\underline{w_{\nu}}}$ for
$\nu\in\mc{OP}_n$.
\end{corollary}

\begin{proof}
The $\bf A$-freeness follows from Lemma~\ref{lem:H-free}. By
Theorem~\ref{thm:spintrace}, the images of $R_{\underline{w_{\nu}}}$
(denoted by $\overline{R}_{\underline{w_{\nu}}}$) for
$\nu\in\mc{OP}_n$ span $\ev{(\sHa/[\sHa,\sHa])}$. By passing to the
field $\K$ and a dimension counting \eqref{eq:dim}, we see that
these elements $\overline{R}_{\underline{w_{\nu}}}$ are linearly
independent. The corollary follows.
\end{proof}

The character table for Hecke algebras or Hecke-Clifford algebras
has a natural generalization for spin Hecke algebra as follows. The
matrix
$$
(\zeta^{\la}_-(R_{\underline{w_{\nu}}})) _{\la\in\mc{SP}_n,\nu\in\mc{OP}_n}
$$
is called the {\em character table} of the spin Hecke algebra $\sHK$
over $\K$. By Corollary~\ref{cor:spinspacetrace} and the linear
independence of irreducible characters $\zeta_-^{\la}$ for
$\la\in\mc{SP}_n$, the character table
$(\zeta^{\la}_-(R_{\underline{w_{\nu}}})) _{\la\in\mc{SP}_n,\nu\in\mc{OP}_n}$ is invertible.

It follows by Corollary~\ref{cor:spinspacetrace} that
$f^-_{\underline{\sigma},\nu}$ in Theorem~\ref{thm:spintrace} is
uniquely determined by $\underline{\sigma}$ and $\nu$. Similar to
\cite{GP1}, $f^-_{\underline{\sigma},\nu}$ will be called the {\em
class polynomials} of spin Hecke algebras. By
Corollary~\ref{cor:spinspacetrace}, there exists a unique function
$f^-_\nu: \sHa \rightarrow \bf A$ characterized by
$$
f^-_{\nu}(R_{\underline{w_{\rho}}})=\delta_{\nu,\rho}, \quad \text{
for }  \rho\in\mc{OP}_n.
$$

By Theorem~\ref{thm:spintrace}, for an arbitrary reduced expression
$\underline{\sigma}$ of $\sigma\in S_n$, we have
\begin{eqnarray*}
f^-_\nu (R_{\underline{\sigma}}) =\left\{
 \begin{array}{ll}
 f^-_{{\underline{\sigma}},\nu}, & \quad \text{for } \ell(\sigma) \text{
 even}, \\
 0 & \quad \text{ otherwise.}
 \end{array}
 \right.
\end{eqnarray*}

Theorem~\ref{thm:spintrace} and Corollary~\ref{cor:spinspacetrace}
imply the following.

\begin{proposition}  \label{prop:sHpoly}
The functions $f^-_{\nu}$ for $\nu\in\mc{OP}_n$ form a basis of the
space of trace functions on $\sHa.$
\end{proposition}
%(The Conjecture~\ref{conj:sHpoly} above is made more precise in
%blue.)

%
\subsection{The trace form $\gimel^-$ on $\sH$}
The trace form $\gimel$ on $\HCK$ induces a symmetrizing trace form,
which will be still denoted by $\gimel$, on $\sHK\otimes\Cl_{n,\K}$
via the isomorphism $\HCK \cong \sHK\otimes\Cl_{n,\K}$, where
$\Cl_{n,\K} =\K \otimes_\C \Cl_n$. This in turn restricts to a
symmetrizing trace form $\gimel^-$ on $\sHK$ (identified with $\sHK
\otimes 1$).

\begin{proposition} \label{prop:vanish}
For any composition $\mu\neq (1^n)$ of $n$,
 $\gimel^-(R_{\underline{w_{\mu}}})=0$.
\end{proposition}

\begin{proof}
By Lemma~\ref{lem:minimal}, it suffices to establish the case when
$\mu\in\mathcal{OP}_n$. We shall prove by induction on the dominance
order of $\mu$.

Observe that the trace form $\gimel$ on $\sH\otimes\Cl_n$ satisfies
that
\begin{equation} \label{eq:tracevalue}
\gimel(T^{\Phi}_{w_{\mu}})=\Big(\frac{v-1}{2} \Big)^{n-\ell(\mu)},
\end{equation}
where we have denoted $T^{\Phi}_{w_{\mu}}=\Phi(T_{w_{\mu}})$. Recall
\eqref{eq:TiPhi} and that
\begin{align*}
w_{\mu} =(s_1s_2\ldots s_{\mu_1-1})(s_{\mu_1+1}\ldots
s_{\mu_1+\mu_2-1}) \cdots(s_{\mu_1+\ldots + \mu_{\ell-1}+1}\ldots
s_{n-1}).
\end{align*}
We write
\begin{equation}  \label{eq:X123}
T^{\Phi}_{w_{\mu}}=X_1 +X_2 +X_3,
\end{equation}
where
\begin{align*}
X_1 =& \Big(-\frac{1}{2}\Big)^{n-\ell(\mu)}R_1(c_1-c_2)
\cdots R_{\mu_1-1}(c_{\mu_1-1}-c_{\mu_1})\cdot\\
&R_{\mu_1+1}(c_{\mu_1+1}-c_{\mu_1+2})\cdots
R_{\mu_1+\mu_2-1}(c_{\mu_1+\mu_2-1}-c_{\mu_1+\mu_2})\cdot \cdots \\
& \cdot R_{\mu_1+\cdots\mu_{\ell-1}+1}(c_{\mu_1+\cdots\mu_{\ell-1}+1}
-c_{\mu_1+\cdots\mu_{\ell-1}+2})\cdots R_{n-1}(c_{n-1}-c_n),
 \\
X_2 =&
 \Big(\frac{v-1}{2}\Big)^{n-\ell(\mu)}(1-c_1c_2)\cdots(1-c_{\mu_1-1}c_{\mu_1})
(1-c_{\mu_1+1}c_{\mu_1+2})\cdots
 \\
 &\cdot (1-c_{\mu_1+\mu_2-1}c_{\mu_1+\mu_2})
 \cdots(1-c_{\mu_1+\ldots+\mu_{\ell-1}+1}c_{\mu_1+\ldots+\mu_{\ell-1}+2})\cdots(1-c_{n-1}c_n),
  \\
X_3 =&\sum_{I\subseteq[n],\ga\in\mc{CP}_n,
(1^n)\neq\overline{\ga}<\mu}a_{\ga, I}R_{\underline{w_{\ga}}}C_I,
\end{align*}
with  $a_{\ga, I}\in\mc{\bf A}$ and $\overline{\ga}$ denoting the
partition corresponding to the composition $\ga$.

It follows by Lemma~\ref{lem:vanish} that
$
\gimel(R_{\underline{w_{\ga}}}C_I)=0
$
if $I\neq\emptyset$, and by Lemma~\ref{lem:minimal} and
induction on $\mu$ by dominance order that
$
\gimel(R_{\underline{w_{\ga}}}) =\gimel^-(R_{\underline{w_{\ga}}}) =0
$
for $\ga\in\mc{CP}_n$ with $\overline{\ga}<\mu$.
Hence
\begin{equation}  \label{eq:X2}
\gimel(X_3) =0.
\end{equation}
Note that $X_2 =(\frac{v-1}{2})^{n-\ell(\mu)} +$  a linear
combination of $C_I$ with $I \neq \emptyset$. By
Lemma~\ref{lem:vanish},
\begin{equation}  \label{eq:X3}
\gimel(X_2) =(\frac{v-1}{2})^{n-\ell(\mu)}.
\end{equation}

For $\mu \in \mathcal{OP}_n$, we have
$$
X_1 =(-\frac{1}{2})^{n-\ell(\mu)}R_{\underline{w_{\mu}}} +  \sum_{\emptyset \neq
I\subseteq[n]}b_{I}R_{\underline{w_{\mu}}}C_I
$$
for some scalars $b_{I}.$ Hence, by Lemma~\ref{lem:vanish}, we have
\begin{equation}  \label{eq:X1}
\gimel(X_1) =(-\frac{1}{2})^{n-\ell(\mu)}
\gimel(R_{\underline{w_{\mu}}}).
\end{equation}
Collecting \eqref{eq:X123}, \eqref{eq:X2}, \eqref{eq:X3} and
\eqref{eq:X1}, we obtain that
$$
\gimel(T^{\Phi}_{w_{\mu}})=\Big(\frac{v-1}{2} \Big)^{n-\ell(\mu)} +
(-\frac{1}{2})^{n-\ell(\mu)} \gimel(R_{\underline{w_{\mu}}}).
$$
By a comparison with \eqref{eq:tracevalue} we conclude that
$\gimel^-(R_{\underline{w_{\mu}}})=\gimel(R_{\underline{w_{\mu}}})=0$.
\end{proof}

By convention, we have $R_{\underline{w_{(1^n)}}}=1$. By
Lemma~\ref{lem:spinminimal}, Theorem~\ref{thm:spintrace} and
Proposition~\ref{prop:vanish}, we have established the following.

\begin{theorem} \label{thm:sHtrace0}
\begin{enumerate}
\item
$\gimel^-(R_{\underline{w_C}}) =0$ for any minimal length
representative $w_C$ in a non-identity conjugacy class $C$ of $S_n$
with any reduced expression $\underline{w_C}$.

\item
$\gimel^-$ is characterized by the property
$\gimel^-(R_{\underline{w_{\nu}}}) =\delta_{\nu, (1^n)}$ for
$\nu\in\mc{OP}_n$.
\end{enumerate}
\end{theorem}

\begin{example}
It is possible that $\gimel^-(R_{\underline{\sigma}}) \neq 0$ if
$\sigma$ is not a minimal length element in its conjugacy class. For
example, the permutation $\sigma=(2,3)(1,4) $ has a reduced
expression $\underline{\sigma}=s_2s_1s_3s_2s_3s_1$, and one computes
that  $\gimel^{-}(R_{\underline{\sigma}})=-(v-1)^4(v^2+1)$.
\end{example}

\begin{remark}
Using the symmetrizing trace function $\gimel^-$ on $\sH$, we can
determine the Schur elements $c^{\la}_-$ associated to the
irreducible characters $\zeta^{\la}_-$ of $\sHK$. These Schur
elements $c^{\la}_-$ turn out to be related to the Schur elements
$c^{\la}$ associated to the  irreducible character $\zeta^{\la}$ of
$\HCK$ (see Theorem~\ref{th:spinSchur}), via
$$
c^{\la}_-=\ds\left\{
\begin{array}{ll}
2^{-k}c^{\la},&\text{ if }n=2k,\\
2^{-k-\delta(\la)}c^{\la},&\text{ if }n=2k+1.
\end{array}
\right.
$$
This can be deduced by using~\eqref{eq:HCspin} and noting that
$\gimel$ can be identified with the tensor product of $\gimel^-$ and
the usual matrix trace on $\Cl_n$.
\end{remark}

%%%
%%%
%%%

\end{document}